\documentclass[psamsfonts]{amsart}

\usepackage{amssymb,amsfonts}
\usepackage[all,arc]{xy}
\usepackage{enumerate}
\usepackage{mathrsfs}
\usepackage[utf8]{inputenc}
\usepackage{hyperref}
\usepackage{mathtools}
\usepackage{tikz}
\usepackage{enumitem}
\newtheorem{thm}{Theorem}[section]
\newtheorem{cor}[thm]{Corollary}

\newtheorem{lem}[thm]{Lemma}

\newtheorem{question}[thm]{Question}
\newtheorem{claim}[thm]{Claim}
\newtheorem{fact}[thm]{Fact}

\theoremstyle{definition}
\newtheorem{definition}[thm]{Definition}

\theoremstyle{remark}
\newtheorem{remark}[thm]{Remark}

\newtheorem{convention}[thm]{Convention}

\newtheorem{notation}[thm]{Notation}

\makeatletter
\let\c@equation\c@thm
\makeatother
\numberwithin{equation}{section}

\title[Tail cone Halpern-Läuchli at a large cardinal]{A tail cone version of the Halpern-Läuchli theorem at a large cardinal}

\author{Jing Zhang}

\newcommand{\Addresses}{{
  \bigskip
  \footnotesize

 \textsc{Department of Mathematical Sciences,\\ Carnegie Mellon University, \\Pittsburgh, Pennsylvania, 15213}\par\nopagebreak
  \textit{E-mail}: \texttt{jingzhang@cmu.edu}
}}

\begin{document}

\begin{abstract}
The classical Halpern-Läuchli theorem states that for any finite coloring of a finite product of finitely branching perfect trees of height $\omega$, there exist strong subtrees sharing the same level set such that tuples in the product of the strong subtrees consisting of elements lying on the same level get the same color.
Relative to large cardinals, we establish the consistency of a tail cone version of the Halpern-Läuchli theorem at a large cardinal (see Theorem \ref{Tail cone homogeneity in full generality}), which, roughly speaking, deals with many colorings simultaneously and diagonally. Among other applications, we generalize a polarized partition relation on rational numbers due to Laver and Galvin to one on linear orders of larger saturation.

\end{abstract}

\maketitle
\tableofcontents
\let\thefootnote\relax
\section{Introduction, definitions and preliminaries}

Recall that a linear order $(X,<)$ is \emph{$\kappa$-saturated} if for any $A, B\subset X$ such that $|A|, |B|<\kappa$ and $A<B$, then there exists $c\in X$ such that $A<c<B$.

Given $\kappa$ regular, define $<_{lex}$ on $2^{<\kappa}$ such that $s<_{lex} t$ iff $s\neq t$ and
\begin{itemize}
\item whenever $s$ and $t$ are incomparable, $s(j)<t(j)$ for $j\in \kappa$ least such that $s(j)\neq t(j)$;
\item whenever $s\sqsubset t$, $t(j)=1$ for $j=dom(s)$;
\item whenever $t \sqsubset s$, $s(j)=0$ for $j=dom(t)$.
\end{itemize}
Note that this ordering is different from the Kleene-Brouwer ordering ($<_{KB}$) in that for any $x\in 2^{<\kappa}$, $x^\frown 0 <_{lex} x <_{lex} x^\frown 1$ while $x^\frown 0<_{KB}<x^\frown 1<_{KB} x $. We follow the convention as in Laver \cite{MR754925} for this ordering while in M. D\v{z}amonja, J. Larson and W. Mitchell \cite{MR2520110}, this ordering is called ''$<_Q$''.

Then $(2^{<\kappa}, <_{lex})$ is $\kappa$-saturated iff $\kappa$ is regular as proved in Lemma 1.7 of \cite{MR2520110}. Note that if $\kappa^{<\kappa}=\kappa$, there is one $\kappa$-saturated linear order of size $\kappa$ up to isomorphism $\mathbb{Q}_\kappa$, whose order type is denoted as $\eta_\kappa$. We will prove the following in this paper.

\begin{thm}\label{main}
Let $d\in \omega$, $\kappa$ inaccessible and $\lambda$ infinite cardinals satisfying that $\lambda\to (\kappa)_{2^\kappa}^{2d}$ (in fact $\lambda = (2^\kappa)^+$ suffices when $d=1$) be given. Suppose further that $\kappa$ is measurable in the forcing extension by $Add(\kappa, \lambda)$. Then \begin{equation}\label{polarizedEquation}
\begin{pmatrix}
\eta_\kappa \\
\vdots \\
\eta_\kappa
\end{pmatrix} \to \begin{pmatrix}
\eta_\kappa \\
\vdots \\
\eta_\kappa
\end{pmatrix}^{\underbrace{1,\cdots, 1}_{d+1}}_{<\kappa, (d+1)!}
\end{equation}
and \begin{equation}
\begin{pmatrix}
\eta_\kappa \\
\vdots \\
\eta_\kappa
\end{pmatrix} \not \to \begin{pmatrix}
\eta_\kappa \\
\vdots \\
\eta_\kappa
\end{pmatrix}^{\underbrace{1,\cdots, 1}_{d+1}}_{<\kappa, (d+1)!-1}
\end{equation}
\end{thm}

The partition relation symbol in (\ref{polarizedEquation}) means the following: for any $\delta<\kappa$ and any $f: \mathbb{Q}_\kappa^{d+1}\to \delta$, there exists $X_j\subset \mathbb{Q}_\kappa$ for $j\leq d$ such that $type(X_j)=\eta_\kappa$ and $|f''\Pi_{j\leq d} X_j|\leq (d+1)!$.

It is known that starting with a ground model satisfying GCH containing a cardinal $\kappa$ that is $(\kappa+2d+1)$-strong (or just $(\kappa+2)$-strong when $d=1$), we can get a forcing extension where the hypothesis in the theorem is true. 

This is a higher analog of a theorem by Laver \cite{MR754925} (the 2-dimensional case is due to Galvin). The classical proof of Laver's theorem is a combination of the Halpern-Läuchli theorem and a fusion type argument. Let us first review some basic definitions. 

\begin{definition}
A partial order $(T,\prec)$ is a \emph{tree} if for each $t\in T$, $\{s\in T: s\prec t\}$ is well ordered by $\prec$.
\end{definition}

\begin{remark}
When there is no confusion about the tree order, we use $T$ to refer to the tree $(T,\prec)$. 
\end{remark}

\begin{definition}
Given $T$ a tree, $t\in T$, $\alpha\in \kappa$ and $A\subset \kappa$,
\begin{itemize}
\item The height of $t$ in $T$ is $ht_T(t)=$ order type of $\{s\in T: s \prec t\}$. When there is no confusion we write $ht$ instead of $ht_T$. $Succ_T(t)=\{s\in T: ht(s)=ht(t)+1 \ \& \  t \prec s\}$, $Pred_T(t)=\{t'\in T: t'\prec t\}$.
\item $T[t]=\{s\in T: s \succeq t\},T(\alpha)=\{s\in T: ht(s)=\alpha\},  T(<\alpha)=\bigcup_{\xi<\alpha} T(\xi), T(\alpha)[t]=\{s\in T(\alpha): s \succeq t \}$.

\end{itemize} 
\end{definition}

\begin{definition}\label{tree}
$(T, \prec)$ is a \emph{nice tree} with the height function $ht_T$ if 
\begin{itemize}
\item $(T,\prec)$ is a tree.
\item $(T, \prec)$ has a unique \emph{root} which we call $root(T)$, that is for any $t\in T$, $root(T)\preceq t$.
\item $T$ is a \emph{$\kappa$-tree}, namely for each $s\in T \ ht_T(s)<\kappa$, $\sup_{t\in T} ht_T(t) =\kappa$ and  for any $\alpha<\kappa$, $|T(\alpha)|<\kappa$. 
\item $T$ is \emph{well-pruned}, that is every maximal path through $T$ has order type $\kappa$ under $\prec$ and for any limit $\gamma$ and $t,s\in T(\gamma)$, if $Pred_T(t)=Pred_T(s)$, then $s=t$.

\item $T$ is \emph{perfect}, namely, for any $t\in T$, there exists $s_0, s_1  \succ t$, $s_0,s_1\in T$, $s_0$ and $s_1$ are incomparable.
\end{itemize}

Note such trees are automatically $<\kappa$-closed, in the following sense: for any $\alpha<\kappa$ and increasing $\bar{s}=\langle s_i\in T: i<\alpha\rangle$, there exists an upper bound $t\in T$ for $\bar{s}$. Also notice that if $\kappa^{<\kappa}=\kappa$ and $T$ is a nice tree, then $(T,<_{lex})\simeq \mathbb{Q}_\kappa$. This fact will be useful later on.

\end{definition}

\begin{definition}\label{StrongSubtree}
Given $T$ a nice tree, we say $T'\subset T$ is a \emph{strong subtree} if $T'$ is a nice tree and there exists a set $A\in [\kappa]^\kappa$ with increasing enumeration $\langle a_i: i<\kappa\rangle$ such that 
\begin{enumerate}
\item for any $\xi <\kappa, \ T'(\xi)\subset T(a_\xi)$
\item for each $\xi<\kappa, s\in T'(\xi)$ and each $t\in Succ_T(s)$, there exists exactly one extension of $t$ in $T'(\xi+1)$.
\end{enumerate}
A set $A$ as above is called a \emph{witnessing level set} for $T'$.
\end{definition}

In search for a generalization of Laver's theorem, two difficulties are present.

The first difficulty is the need for the Halpern-Läuchli theorem at a large cardinal. The classical Halpern-Läuchli theorem is a partition theorem about finite products of finitely branching infinite perfect trees of height $\aleph_0$. It was first proved in \cite{MR0200172} by J. Halpern and H. Läuchli and was used in \cite{MR0284328} to establish the independence of the axiom of choice from the Boolean prime ideal theorem over ZF. 

L.~Harrington came up with a nice forcing proof of the classical Halpern-Läuchli theorem, see for example  \cite{MR1486583}. S. Shelah \cite{MR1218224} generalized L. Harrington's argument to show that consistently for any $m\in \omega$ and any $<\kappa$ coloring of $\bigcup_{\xi<\kappa} [2^{\xi}]^m$,  there exists a strong subtree $T\subset 2^{<\kappa}$ such that for any $\xi<\kappa$, $[T\cap 2^\xi]^m$ only gets finitely many colors. M. D\v{z}amonja, J. Larson and W. Mitchell in \cite{MR2520110} further extended the result to deal with any $m$-sized antichains in $T$. They further applied the theorem to obtain a partition theorem for large saturated linear orders, generalizing a classical theorem by D.~Devlin \cite{Devlin}. In particular they show: 
\begin{thm}[M. D\v{z}amonja, J. Larson and W. Mitchell \cite{MR2520110}]\label{unpo}
Under the hypothesis of Theorem \ref{main} for $d\geq 2$, there exists $t_d^+\in \omega$ such that $(\eta_\kappa)\to (\eta_\kappa)_{<\kappa, t_d^+}^d$ and $(\eta_\kappa)\not\to (\eta_\kappa)_{<\kappa, t_d^+-1}^d$. In fact $[\eta_\kappa]^d$ can be classified into $t_d^+$ many types and the coloring restricted on each type is constant. Here $t_d^+\geq t_d + (2^{d-1})(-1+ \Pi_{i<d}i!)$ where $t_d$ is the $d$-th tangent number that can be calculated recursively as $t_n = \Sigma_{i=1}^{ n-1} {2n-2 \choose 2i-1} t_i t_{n-i}$ with $t_1=1$.
\end{thm}
Notice we can derive a version of Theorem \ref{main} from the above theorem. Given $f: (2^{<\kappa})^d\to \delta$, form $g: [2^{<\kappa}]^d \to \delta$ such that for any $\sigma_0<_{lex} \cdots <_{lex} \sigma_{d-1}$, 
$g(\{\sigma_0,\cdots, \sigma_{d-1}\})=f(\sigma_0,\cdots, \sigma_{d-1})$. Find $T\subset 2^{<\kappa}$ with $(T,<_{lex})$ of type $\eta_\kappa$ such that $|g''[T]^d|\leq t_d^+$. Find sub-orders $T_0 <_{lex}\cdots <_{lex} T_{d-1}$ of $T$ each of type $\eta_\kappa$, then $|f(T_0\times \cdots \times T_{d-1})|\leq t_d^+$. 
However the conclusion is not exactly that of Theorem \ref{main} as 
\begin{enumerate}
\item $t_d^+$ is a lot larger than $d!$ when $d$ is large;
\item under the hypothesis with parameter $d$, the conclusion is valid for exponent $j$ for every $j\leq d$ rather than for exponent $j$ for every $j\leq d+1$ as in Theorem \ref{main}.
\end{enumerate}
As the canonical types that occur in the polarized case are simpler, the method presented here will be different from that in \cite{MR2520110} and closer to Laver's argument in \cite{MR754925}. 

More recently in \cite{DobrinenHathaway}, N. Dobrinen and D. Hathaway considered the polarized version of the Halpern-Läuchli theorem regarding trees whose height is some large cardinal $\kappa$. They prove the equivalence of various forms of the Halpern-Läuchli theorem at the level of weakly compact cardinals and go on to establish the consistency of the asymmetric dense set version of the Halpern-Läuchli theorem (Theorem \ref{DHmainThm}). It turns out that this polarized version is the most relevant for our purpose.

The second difficulty in generalizing Laver's theorem is the lack of fusion arguments when the height of the tree is larger than $\aleph_0$. One might hope to find some version of the Halpern-Läuchli theorem that finds ''large'' strong subtrees. For example, in the countable case we have the following: given a Ramsey ultrafilter $U$ on $\omega$, a finitely branching perfect tree $T$ of height $\omega$ and a finite coloring $f$, there exists a strong subtree $T'$ whose level set $A\in U$ such that $T'$ is homogeneous with respect to $f$. In fact higher dimensional results are also true, which is a corollary of a theorem by Mathias \cite{MR0491197}, see \cite{MR1486583} for a proof. One might hope for the analogous result for $\kappa$ measurable with normal measure $U$, which will then give us ''large'' strong subtrees. However, the analog is false as witnessed by the Sierpinski-style coloring $2^{<\kappa}\to 2$ ensuring that for each strong limit $\alpha<\kappa$ and each subtree $T\subset 2^{<\alpha}$ with $2^{\alpha}$ branches there exist two branches, i.e. nodes on level $\alpha$, that get different colors. This shows that in general the witnessing level sets cannot even be stationary.

In order to tackle these two problems, we consider a version of the Halpern-Läuchli theorem, which we call the \emph{tail cone version}, such that it incorporates the consequence we want to achieve with the usual fusion argument. We will establish its consistency.

We will mainly be concerned with the versions of the Halpern-Läuchli theorem considered in \cite{DobrinenHathaway}.
In Section \ref{1-D}, we show at the level of a weakly compact cardinal, the 1-dimensional Halpern-Läuchli theorem holds for $<\kappa$ many colors, extending a previous result by Dobrinen and Hathaway \cite{DobrinenHathaway} where the number of colors is finite. In Section \ref{TailCone}, we establish the consistency of the tail cone version of the Halpern-Läuchli theorem. Further we show there is a dimension boost phenomenon, namely, the $d$-dimensional tail cone Halpern-Läuchli theorem implies the usual $(d+1)$-dimensional Halpern-Läuchli theorem. Then in Section \ref{polarize} we prove the tail cone version of the Halpern-Läuchli theorem implies the polarized partition relation in Theorem \ref{main}. Finally in Section \ref{NonImp}, we show the consistency of the Halpern-Läuchli theorem as considered in \cite{DobrinenHathaway} being true at $\kappa$ yet $\kappa$ is not weakly compact, in contrast with other versions previously considered.

We end the section with more definitions, notations and preliminaries. Let $\kappa$ be a strongly inaccessible cardinal.

\begin{notation}
We will typically use a letter with an overhead bar to denote a tuple and a subscript to indicate the element in the tuple with the prescribed coordinate.
For example, if $\langle A_j: j\in I\rangle$ is a given set with an index set $I$, then $\bar{x}\in \Pi_{j} A_j$ and $x_j$ denotes the $j$-th element in the tuple. In particular, $x_j\in A_j$.
\end{notation}

\begin{definition}[\cite{DobrinenHathaway}, \cite{MR2603812}]
Given a sequence of nice trees $\langle T_j: j<d\rangle$, $M\subset \Pi_{j<d}T_j$, and $\bar{x}\in \Pi_{j<d}T_j$ with $\xi>\sup_{j<d}ht_{T_j}(x_j)$, we say $M$ is \emph{$\bar{x}$-$\xi$-dense} if for each $\bar{t}\in \Pi_{j<d} T_j (\xi) [x_j]$, there exists $\bar{t}'\in M$ such that $t'_j\succeq_{ T_j} t_j$ for each $j<d$. $M$ is a \emph{somewhere dense set} if there exists $\bar{x}$ and $\xi$ such that $M$ is $\bar{x}$-$\xi$-dense. A \emph{somewhere dense matrix} is a somewhere dense set of the form $M_0 \times \cdots \times M_{d-1}$ where $M_j\subset T_j$. 
We usually add an overhead bar to denote a matrix, that is $\bar{M}$ usually means $\bar{M}=M_0\times \cdots M_{d-1}$ for some $M_j\subset T_j$ for $j<d$.
\end{definition}

We call $\bar{t}\in \Pi_{j<d} T_j$ a \emph{level sequence} and $\bar{M}$ a \emph{level matrix} if all the elements lie on the same level. The height of a level sequence (a level matrix) will be defined naturally as the common level on which each element lies, denoted as $ht(\bar{t})$ ($ht(\bar{M})$).

\begin{convention}\label{RestrictionConvention}
Whenever $\bar{t}\in \Pi_{j<d} T_j$ is a level sequence, for some $\xi\leq ht(\bar{t})$, we use $\bar{t}\restriction_{\Pi_{j<d} T_j} \xi$ to denote the tuple whose $j$-th coordinate is the $\xi$-th element with respect to the natural tree ordering in $\{s\in T_j: s\leq_{T_j} t_j\}$.  Define similarly for level matrices.
\end{convention}

\begin{notation}\label{restriction}
In this paper, we will frequently take restrictions with respect to different subtrees. To avoid some cumbersome indices, we will adopt the following notation. The superscript for the restriction symbol will correspond to taking restriction with respect to the strong subtrees with the same superscript. For example, given $T'_j$, which are strong subtrees of $T_j$ for $j<d$ and a level sequence $\bar{t}\in \Pi_{j<d} T'_j$, $\bar{t}\restriction'  i$ denotes $\bar{t}\restriction_{\Pi_{j<d}T_j'} i$. Likewise $\restriction^*$ for $T_j^*$, $\restriction^i$ for $T_j^i$ etc.
Similarly for level matrices.
\end{notation}

The next definition collects different versions of the Halpern-Läuchli theorem that will be considered in this paper. The motivation for the following different versions considered here is that the asymmetric versions are technical strengthenings of the Halpern-Läuchli theorem while the somewhere dense version is a localized statement that is equivalent under suitable assumptions on the large cardinal $\kappa$.
The dense set version enables the construction of desired strong subtrees in a rather straightforward way.

\begin{definition}[Dobrinen and Hathaway \cite{DobrinenHathaway}]\label{HLdef}

Let $d\in \omega$, $\kappa$ inaccessible and $\delta<\kappa$ be given.

\begin{enumerate}
\item $HL(d,\delta,\kappa)$ (the strong subtree version) abbreviates the following statement:
For any $\langle T_j: j<d\rangle$, a sequence of nice trees, and any coloring $f: \Pi_{j<d} T_j \to \delta$, there exist strong subtrees $T_j'\subset T_j$ with the same witnessing level set such that 
$f'' \bigcup_{\xi<\kappa}\Pi_{j<d} T'_j(\xi)$ is constant.

\item $HL^{asym}(d,\delta, \kappa)$ (the asymmetric strong subtree version) abbreviates the following statement:
For any $\langle T_j: j<d\rangle$, a sequence of nice trees, and any coloring $f: \Pi_{j<d} T_j \to \delta$, there exist strong subtrees $T_j'\subset T_j$ with root $t_j$ sharing the same witnessing level set such that for some $\gamma<\delta$, 
$f'' \bigcup_{\xi<\kappa}\Pi_{j<d} T'_j(\xi)=\{\gamma\}$ and if $\gamma=0$ we could take $t_j=root(T_j)$ for each $j<d$.

\item $DSHL(d,\delta,\kappa)$ (the dense set version) abbreviates the following statement: for any nice trees $\langle T_j: j<d\rangle$ and $f: \Pi_{j<d}T_j \to \delta$, there exist $\xi\in \kappa$, $\bar{t}\in \Pi_{j<d} T_j(\xi)$ and $\gamma<\delta$ such that for all $\eta\in \kappa$ there exists a $\bar{t}$-$\eta$-dense level matrix $\bar{D}\subset \Pi_{j<d} T_j$ and $f'' \bar{D}=\{\gamma\}$.

\item $DSHL^{asym}(d,\delta,\kappa)$ (the asymmetric dense set version) abbreviates the following statement: for any nice trees $\langle T_i: i<d\rangle$ and $f: \Pi_{i<d}T_i \to \delta$, there exist $\xi\in \kappa$, $\bar{t}\in \Pi_{i<d} T_i(\xi)$ and $\gamma<\delta$ such that for all $\eta\in \kappa$ there exists a $\bar{t}$-$\eta$-dense level matrix $\bar{D}\subset \Pi_{j<d} T_j$ and $f'' \bar{D}=\{\gamma\}$. In addition if $\gamma$ is $0$, then $\bar{t}$ could be chosen to be $\langle root(T_j):j<d  \rangle$.

\item $SDHL(d,\delta,\kappa)$ (the leveled somewhere dense version) abbreviates the following statement: for any nice trees $\langle T_j: j<d\rangle$ and $f: \Pi_{j<d}T_j \to \delta$, there exists a level sequence $\bar{t}\in \Pi_{j<d} T_i$ and some $\bar{t}$-$(ht(\bar{t})+1)$-dense level matrix $\bar{D}$ that $|f'' \bar{D}|=1$. 
\item $SDHL'(d,\delta,\kappa)$ (the somewhere dense version) abbreviates the following statement: for any nice trees $\langle T_j: j<d\rangle$ and $f: \Pi_{j<d}T_j \to \delta$, there exists a somewhere dense matrix $\bar{D}$ that $|f'' \bar{D}|=1$. 
\end{enumerate}
\end{definition}

Note $HL^{asym}(d,\delta,\kappa)$ and $DSHL^{asym}(d,\delta,\kappa)$ imply $HL(d,\delta,\kappa)$ and $DSHL(d,\delta,\kappa)$ respectively.

It is important in the strong subtree version in Definition \ref{HLdef} that these strong subtrees share the same level set. However, there are circumstances, for example in the proof of Theorem \ref{DimInduct} and Lemma \ref{AlmostAll}, where strong subtrees not necessarily sharing the same level set are produced. Hence whenever a sequence of strong subtrees is produced, we will always make it clear whether they share the same level set.

\begin{thm}[Dobrinen and Hathaway \cite{DobrinenHathaway}]\label{EquiForm}
For $\kappa$ a weakly compact cardinal, $d\in \omega, \delta<\kappa$, the following are equivalent.
\begin{enumerate}

\item $DSHL(d,\delta,\kappa)$

\item $HL(d,\delta,\kappa)$

\item $SDHL(d,\delta,\kappa)$
\item $SDHL'(d,\delta,\kappa)$

\end{enumerate}

In fact (1), (2), (3) are already equivalent when $\kappa$ is merely strongly inaccessible.
\end{thm}

\begin{thm}[Dobrinen and Hathaway \cite{DobrinenHathaway} Theorem 4.6]\label{DHmainThm}
Let $d\in \omega$ and let $\kappa<\lambda$ be cardinals such that $\lambda\to (\kappa)^d_\kappa$. Suppose further that $\kappa$ is measurable in the forcing extension by $Add(\kappa,\lambda)$. Then in $V$, $DSHL^{asym}(d,\delta, \kappa)$ holds for all $\delta<\kappa$.
\end{thm}

It is pointed out in \cite{DobrinenHathaway} that the hypothesis can be obtained starting from $\kappa$ $(\kappa+d$)-strong.

The notion of branches and the bounded topology on the branches will be useful in proving $d$-dimensional tail cone homogeneity implies the $d+1$-dimensional Halpern-Läuchli theorem.

\begin{definition}\label{density}
Given a nice tree $T$, and assume $T\subset \kappa$. Then $[T]$ denotes the collection of branches through $T$, namely $\{f \in \Pi_{\xi<\kappa} T(\xi): \forall \alpha<\beta<\kappa \  f(\alpha)\prec_T f(\beta)\}$. We can topologize $[T]$ with basic open sets $N^T_t=\{f\in [T]: t= f(ht_T(t))\}$ where $t\in T$. We call this topology $\tau$. For any $A\subset [T]$, we use $\exists^\tau x\in A$ to mean ``there exists a collection of elements in $A$ that is $\tau$-dense in $A$''.
\end{definition}

\section{The 1-dimensional asymmetric version of the Halpern-Läuchli theorem at a weakly compact cardinal for less than $\kappa$ many colors}\label{1-D}

In \cite{DobrinenHathaway}, $HL(1,k,\kappa)$ is established for any finite $k$ and any weakly compact $\kappa$. We pointed out that their proof works for any strongly inaccessible $\kappa$.

In the context of weakly compact cardinals, we are able to extend this result to fewer than $\kappa$ many colors.

\begin{thm}\label{WeakCompactness}
If $\kappa$ is weakly compact, then for all $\delta<\kappa$, $HL^{asym}(1,\delta,\kappa)$ holds.
\end{thm}
\begin{proof}
Given $\delta<\kappa$ and a coloring $f: T\to \delta$ where $T$ is a nice tree, let $A_{q,\gamma}=\{\alpha<\kappa:  \exists q'\in T(\alpha)[q] \ f(q')=\gamma\}$ and consider the collection $\mathcal{X}=\{A_{q,\gamma} \subset \kappa: q\in T, \gamma<\delta\}$.

Notice that $|\mathcal{X}|\leq \kappa$, hence by weak compactness, there exists a non-principal $\kappa$-complete filter $\mathcal{F}$ on $\kappa$ that decides every member of $\mathcal{X}$, that is for each $A\in \mathcal{X}$, exactly one of $A$, $A^c$ is in $\mathcal{X}$. Notice that for any $\gamma\in \delta$ if $p \preceq  q$ then $A_{q,\gamma}\in \mathcal{F}$ implies $ A_{p,\gamma}\in \mathcal{F}$.  In other words, if $p\preceq q$ and $A_{p,\gamma}^c\in \mathcal{F}$, then $A_{q,\gamma}^c\in \mathcal{F}$.

We show the following claim:
\begin{claim}\label{ImportantClaim}
For each $p\in T$ there exist $\gamma\in \delta$ and $q \succeq p$ such that for all $q' \succeq q, \ A_{q',\gamma}\in \mathcal{F}$.
\end{claim}

\begin{proof}[Proof of the claim]
Suppose not, then there exists $p\in T$ such that for all $\gamma\in \delta$ and for all $q\succeq p$ there exists $q'\succeq q$ such that $A_{q',\gamma} \not\in \mathcal{F}$ (so $A_{q',\gamma}^c \in \mathcal{F}$).
Inductively we build a chain of nodes $\langle q_\gamma: \gamma<\delta\rangle$ such that
\begin{itemize}
\item for any $\alpha<\beta<\delta, q_{\alpha}\prec q_{\beta}$;
\item for any $\gamma\in \delta$, $A_{q_\gamma, \gamma}^c\in \mathcal{F}$.
\end{itemize}
Let $q_{-1}=p$. Inductively at stage $\alpha<\delta$ suppose we already have a $\langle q_{i}: i<\alpha\rangle $, let $q'$ be some upper bound for this sequence in $T$. By hypothesis, we can extend $q'$ to $q_\alpha$ such that $A_{q_\alpha,\alpha}^c\in \mathcal{F}$.
Hence the construction of such sequence is possible. Now let $q^*$ be some upper bound for $\langle q_{\alpha}: \alpha<\delta\rangle$. We know that for any $\gamma\in \delta,  A_{q^*, \gamma}^c\in \mathcal{F}$. By the $\kappa$-completeness of $\mathcal{F}$ we can find a $\beta\in \bigcap_{\gamma\in \delta} A^c_{q^*, \gamma}\backslash ht(q^*)+1$. On level $\beta$, for any $t\in T(\beta)[q^*]$ and any $\gamma<\delta$, $f(t)\neq \gamma$. However by the fact that the tree is well-pruned, there must exist some node on level $\beta$ extending $q^*$ which gets some color in $\delta$, which is a contradiction. This proves the claim.
\end{proof}

Let $p=root(T)$. We will pick $\gamma\in \delta$ and build a strong subtree $T'$ of color $\gamma$.
Consider the following cases:
\begin{enumerate}

\item[Case 1] For all $q'\in T[p]$, $A_{q', 0}\in \mathcal{F}$. Then $\gamma=0$ and $root(T)=root(T')$.

\item[Case 2] Case 1 is not true. Let $q^*\in T$ be such that $A_{q^*, 0}^c\in \mathcal{F}$. Note that this is also true for all extensions of $q^*$ in $T$. By Claim \ref{ImportantClaim}, we can find $q\in T[q^*]$ and $\gamma'\in \delta-\{0\}$ such that for all $q'\in T[q], \ A_{q',\gamma'}\in \mathcal{F}$. Let $q=root(T')$ and $\gamma=\gamma'$.
\end{enumerate}

With $root(T')$ and $\gamma$ we can recursively build a strong subtree $T'$ with level set $A=\{a_i: i<\kappa\}$ as follows.

At stage $\alpha+1$, assume we have constructed $T'(\leq \alpha)$. Enumerate the immediate successors of $T'(\alpha)$ in $T$ as $\{t_k: k<\beta\}$. Hence they lie on $T(a_\alpha+1)$. As $\mathcal{F}$ is $\kappa$-complete, we can find $\eta\in \bigcap_{k<\beta} A_{t_k, \gamma}\backslash a_\alpha+1$. Find extensions $t_k'$ of $t_k$ in $T(\eta)$ of color $\gamma$. Let $T'(\alpha+1)=\{t_k':k<\beta\}$ and $a_{\alpha+1}=\eta$.

At stage $\nu$ where $\nu$ is a limit, let $\{t_k: k<\beta\}$ enumerate nodes in $T(\bar{\nu})$ where $\bar{\nu}=\sup_{i<\nu} a_i$ such that each $t_k$ is on top of a branch through $T'(<\nu)$. Repeat the procedure as before to define $T'(\nu)$ and $a_\nu$.

It is easy to see from the construction that $T'$ is a strong subtree with level set $A$ that gets color $\gamma$ and if $\gamma$ is $0$, then $root(T')=root(T)$. 

\end{proof}

\begin{remark}\label{SimultaneousVersion}
We can use the same proof to get the following simultaneous/diagonal version under the same assumptions:

 For $\gamma<\kappa$, $\langle \delta_j<\kappa: j<\kappa\rangle$, a sequence of nice trees $\langle T_j: j<\kappa \rangle$ and a sequence of colorings $\langle f_j:  T_j\to \delta_j | \ j<\kappa\rangle$ there exists $A\in [\kappa]^\kappa$ enumerated increasingly as $\{a_i: i<\kappa\}$ such that for all $j<\kappa$, there exists a strong subtree $T_j'\subset T_j$ with witnessing level set $A\backslash a_j+1$ and $f_j\restriction T_j'$ is constant.

\end{remark}

\section{Tail cone homogeneity}\label{TailCone}

In this section, we establish the consistency of the following analog of a consequence of a typical fusion argument in the classical context. The following argument makes use of a combinatorial lemma due to Shelah (see \cite{MR1218224} and also a nice presentation in \cite{MR2520110}) who generalized a forcing argument by Harrington in the countable context.

\begin{thm}\label{Tail cone homogeneity in full generality}
Let $d\in \omega$, and let $\kappa, \lambda$ be cardinals such that
\begin{enumerate}
\item $\lambda$ satisfies $\lambda\to (\kappa)_{2^\kappa}^{2d}$ (when $d=1$, $\lambda=(2^\kappa)^+$ suffices);
\item $\kappa$ is inaccessible and is measurable in the forcing extension by $\mathbb{P}=Add(\kappa,\lambda)$.
\end{enumerate} 
 Then in $V$,
for any $\langle \delta_i<\kappa: i<\kappa\rangle$, $\langle T_j: j<d\rangle$ nice trees and $f_i: \Pi_{j<d}T_j\to \delta_i$ with $i<\kappa$, there exists a sequence of strong subtrees $\langle T'_j\subset T_j: j<d\rangle$ sharing the same witnessing level set such that for each $i<\kappa$, $\xi\geq i+1$, and $\bar{t}\in \Pi_{j<d}T'_j(\xi)$, $f_i(\bar{t})=f_i(\bar{t}\restriction' i+1)$. In other words, for each $i<\kappa$ and $\langle y_j: j<d\rangle\in \Pi_{j<d} T_j'(i+1)$, $f_i'' \bigcup_{i+1\leq \xi <\kappa} \Pi_{j<d} T'_j(\xi)[y_j]$ is a singleton.

\end{thm}

\begin{remark} Use $HL^{tc}(d,<\kappa,\kappa)$ to denote the conclusion of the theorem, which we call \emph{tail cone homogeneity}.
If there exists $\delta<\kappa$ such that each coloring uses at most $\delta$ many colors, then the conclusion of the theorem is abbreviated as $HL^{tc}(d,\delta,\kappa)$.

\end{remark}

\begin{remark}
$HL^{tc}(d,<\omega,\omega)$ can be proved using $HL(d,<\omega, \omega)$ combined with a fusion type argument. Note also that, since $\kappa=\omega$ satisfies the hypothesis in Theorem \ref{Tail cone homogeneity in full generality}, our proof yields $HL^{tc}(d,<\omega,\omega)$ in ZFC.
\end{remark}

Without loss of generality we may assume $T_j, j<d$ are mutually disjoint. $\mathbb{P}$ consists of partial functions $p: \lambda \to \Pi_{j<d} T_j$ such that $|support(p)|<\kappa$. $\mathbb{P}$ is ordered by the following: $q\leq_\mathbb{P} p$ (namely, $q$ is stronger than $p$) iff $support(q)\supset support(p)$ and for each $\gamma\in support(p)$, $q(\gamma)(j)\succeq_{T_j} p(\gamma)(j)$ for any $j<d$. What this forcing is doing is to add $\lambda$ generic tuples of branches through $\Pi_{j<d} T_j$. For each $\eta<\lambda$, let $\dot{c}_\eta$ be the $\mathbb{P}$-name for the generic $d$-tuple of branches at $\eta$-th coordinate, namely, in $V[G]$, $(\dot{c}_\eta)_G = G(\eta)\in \Pi_{j<d} [T_j]$.

We claim that $\mathbb{P}$ is forcing equivalent to $Add(\kappa, \lambda)$. To see this, since $\mathbb{P}$ is the $<\kappa$-support product of $\lambda$ many copies of $\Pi_{j<d} T_j$ with the coordinate-wise tree order, it suffices to show each copy in the product is equivalent to $Add(\kappa,1)$. But each of these forcings is separative, $\kappa$-closed, of cardinality $\kappa$, and hence is equivalent to $Add(\kappa, 1)$. For a proof see \cite{MR2768691} for example.  Notice that for any subset $A\subset \mathbb{P}$ of cardinality $<\kappa$ such that conditions in $A$ are pairwise compatible, then there exists in $\mathbb{P}$ a greatest lower bound for elements in $A$.

\begin{definition}
For $B \subset \lambda$, $p\in \mathbb{P}$, define $\mathbb{P}\restriction B$ to be $\{p\in \mathbb{P}: support(p)\subset B\}$ and $p\restriction B $ is the condition that agrees with $p$ on $B$ and whose support is contained in $B$.
\end{definition}

\begin{definition}
Given $W_0, W_1\subset \lambda$ of the same order type, let $h_{W_0, W_1}$ be the unique order preserving map from $W_0$ to $W_1$. We can then induce a copying action $h_{W_0, W_1}^\mathbb{P}: \mathbb{P}\restriction W_0 \to \mathbb{P}\restriction W_1$ such that $\forall p\in \mathbb{P}\restriction W_0, j\in W_0  \  h_{W_0, W_1}^\mathbb{P} (p)(h_{W_0, W_1}(j))=p(j)$. With a slight abuse of notation, we write $h_{W_0, W_1}^\mathbb{P}$ as $h_{W_0, W_1}$ since it can be easily inferred from the context.
\end{definition}

The following combinatorial lemma plays a key role in the construction.

\begin{lem}[Lemma 4.1 in \cite{MR1218224} and Claim 7.2.a in \cite{MR2520110} with slight modifications, see the remarks that follow]\label{Cleanup}
Suppose $\lambda \to (\kappa)_{2^\kappa}^{2d}$ and $\langle \delta_i: i<\kappa\rangle$ is a sequence of cardinals $<\kappa$. Let $\langle \dot{\tau}_i(u): i<\kappa, u\in [\lambda]^d \rangle$ be a collection of $\mathbb{P}$-names such that each $\dot{\tau}_i(u)$ is a name for an ordinal $<\delta_i$. Then there exists $E\in[\lambda]^\kappa$ and $W: [E]^{\leq d}\to [\lambda]^{\leq \kappa}$ such that 
\begin{enumerate}[label=\textbf{CL.\arabic*}]
\item \label{CL1} For all $u\in [E]^d$, $u\subset W(u)$ and $\mathbb{P}\restriction W(u)$ contains a maximal antichain deciding the value of $\dot{\tau}_i(u)$ for each $i<\kappa$.
\item \label{CL2} For any $u,v\in [E]^d$, $type(W(u))=type(W(v))$, $h_{W(u), W(v)}(u)=v$ and for any $p\in \mathbb{P}\restriction W(u)$, for any $i<\kappa$ and $\gamma<\delta_i$, $p\Vdash \dot{\tau}_i(u)=\gamma \Leftrightarrow h_{W(u), W(v)}(p)\Vdash \dot{\tau}_i(v)=\gamma$.
\item \label{CL3}  For any $u, v \in [E]^{d}$, $W(u)\cap W(v)=W(u\cap v)$.
\item \label{CL4}  For any $u_1\subset u_2, u'_1\subset u'_2 $ where $u_2, u_2'\in [E]^{d}$, if $(u_2, u_1,<)\simeq (u_2', u_1', <)$, then $W(u_1')=h_{W(u_2), W(u_2')}( W(u_1))$.

\end{enumerate}
\end{lem}

\begin{remark}
Lemma 4.1 appeared in \cite{MR1218224} where for each $u\in [\lambda]^d$ the size of $\{\dot{\tau}_i(u): i<\kappa\}$ is at most countable while in Claim 7.2.a in \cite{MR2520110} the requirement (4) is not present and there exists a uniform upper bound $\delta<\kappa$ for $\langle \delta_i: i<\kappa\rangle$. However the lemma with these modifications still follow from the proof presented in \cite{MR2520110}. For the sake of completeness, we include a sketch of this lemma in the Appendix \ref{AddedProof}.
We point out that when $d=1$, $\lambda= (2^\kappa)^+$ suffices. Note that $(2^\kappa)^+$ is smaller than the least cardinal $\lambda$ satisfying $\lambda \to (\kappa)_{2^\kappa}^2$. 

\end{remark}

\begin{proof}[Proof of Theorem \ref{Tail cone homogeneity in full generality}]

Recall that for each $\xi<\lambda$, $\dot{c}_\xi$ is a name for the $\xi$-th generic tuple of branches through $\Pi_{j<d} T_j$.
Fix $\{u_0<\cdots < u_{d-1}\}=u\in [\lambda]^d$ and $i<\kappa$. In $V[G]$, let $D$ be a $\kappa$-complete ultrafilter on $\kappa$. 
Hence there exists a unique color $\tau_i(u)<\delta_i$
such that
 \begin{equation}\label{property}
 A_u^i :=\{\alpha<\kappa: f_i(c_{u_0}(0)\restriction \alpha, \cdots, c_{u_{d-1}}(d-1)\restriction \alpha)=\tau_i(u)\}\in D.
\end{equation}

 Let $\dot{D}$, $\dot{\tau}_i(u)$ and $\dot{A}^i_u$ be the respective names such that the relevant facts above are forced by $1_\mathbb{P}$.

Apply Lemma \ref{Cleanup} to get the desired $E$ and $W$ with respect to $\langle \dot{\tau}_i(u): i<\kappa, u\in [\lambda]^d\rangle$. Let an injective association $\{e_s\in E: s\in \bigcup_{j<d} T_j\}$ be such that: 

\begin{equation}\label{ordertype}
\text{for each } \xi<\kappa, j_0< j_1<d, s_{0}\in T_{j_0}(\xi) \text{ and } s_{1}\in T_{j_1}(\xi),  \text{we have } e_{s_{0}}<e_{s_{1}}.
\end{equation}

Given any $\bar{u}\subset \bigcup_{j<d} T_j$, we use $e_{\bar{u}}$ to denote $\{e_{s}:  s\in u\}$. If $\bar{u}\in \Pi_{j<d} T_j$, then use $e_{\bar{u}}$ to denote $\{e_{s}: \text{for some } j<d,  u_j=s\}$.

We will recursively construct the desired strong subtrees $T_j'\subset T_j, j<d$ with the same witnessing level set $A=\{a_i: i<\kappa\}$. Along the way we will define $\{p_{\bar{u}}\in \mathbb{P}\restriction W(e_{\bar{u}}): \bar{u}\in \bigcup_{i<\kappa}\Pi_{j<d} T'_j(i)\}$ and $\{j_{i,\bar{v}}<\delta_i: \bar{v}\in \Pi_{j<d} T'_j(i+1), i<\kappa\}$ maintaining the following construction invariants: 
\begin{enumerate}[label=\textbf{CI.\arabic*}]
\item \label{1} For any $i<\kappa, \bar{v}\in \Pi_{j<d}T'_j(i+1)\subset \Pi_{j<d}T_j(a_{i+1})$, $p_{\bar{v}} \Vdash \dot{\tau}_i(e_{\bar{v}})=j_{i,\bar{v}}$  and  for any level sequence $\bar{v}'\in \Pi_{j<d}T'_j$ with $\bar{v} \prec \bar{v}'$, $p_{\bar{v}'} \Vdash \dot{\tau}_i (e_{\bar{v}'})=j_{i, \bar{v}}$ and $f_i(\bar{v}')=j_{i,\bar{v}}$;
\item \label{2} For any $i<\kappa$, $\bar{u}\in \Pi_{j<d}T'_j(i)$, for each $j<d$, $p_u(e_{u_j})(j)=u_j$ and $(\forall z\in dom(p_{\bar{u}})) (\forall j<d) \ ht_{T_j}(p_{\bar{u}}(z, j))\leq a_i$ (\emph{We often times write $p_{\bar{u}}(a)(b)$ as $p_{\bar{u}}(a,b)$ by currying}); 
\item \label{3} For any $i<\kappa$, $\{p_{\bar{u}}: \bar{u}\in \Pi_{j<d}T'_j(i)\}$ consists of pairwise compatible conditions, hence has a greatest lower bound;
\item  \label{4} For any level sequences $\bar{u}, \bar{u}'$ in $\Pi_{j<d} T_j'$ with $\bar{u}\prec \bar{u}'$, $p_{\bar{u}'} \leq_{\mathbb{P}}  h_{W(e_{\bar{u}}), W(e_{\bar{u}'})}(p_{\bar{u}})$.
\end{enumerate}

The \emph{key point} of this construction is that by \ref{1} we are ensuring that for any $i<\kappa$, any level sequence $\bar{v}\in \Pi_{j<d} T_j'(i+1)$, any level sequence $\bar{v}'$ above $\bar{v}$ in $\Pi_{j<d} T_j'$ gets color $j_{i,\bar{v}}$ under $f_i$. This verifies the desired tail cone homogeneity we are asking for. The rest of the invariants are technical and are intended for the induction to go through.

To start the construction, let $root(T_j')=root(T_j)$ for all $j<d$ and $a_0=0$.

\textbf{Stage $\alpha+1$}: Suppose we have defined 
$T'(\leq \alpha)$ and $\langle a_i: i\leq \alpha\rangle$.
For each $j<d$ let $B_j=\{t^j_k: k<\beta_j\}$ enumerate the immediate successors of the nodes on $T_j'(\alpha)$ in $T_j$. Thus $B_j$ is a subset of $T_j(a_\alpha+1)$.

We remind the reader of Notation \ref{restriction} as restrictions with respect to different trees will be taken in the following construction.

Let $\bar{t}\in \Pi_{j<d} B_j$ and $\bar{s}\in \Pi_{j<d} T'_j(\alpha)$ such that $\bar{s}\prec \bar{t}$.

\begin{claim}\label{3.10}
$\{h_{W(e_{\bar{s}\restriction' \beta}), W(e_{\bar{t}})}(p_{\bar{s}\restriction' \beta}): \beta \leq \alpha\}$ consists of pairwise compatible conditions, and in fact it is decreasing as $\beta$ increases.
\end{claim}
\begin{proof}[Proof of the claim]
Notice by \ref{4} of the construction invariants we have for any $\beta<\alpha$, $p_{\bar{s}}\leq h_{W(e_{\bar{s}\restriction' \beta}), W(e_{\bar{s}})}(p_{\bar{s}\restriction' \beta})$, hence 
\begin{equation*}
h_{W(e_{\bar{s}}), W(e_{\bar{t}})}(p_{\bar{s}}) \leq h_{W(e_{\bar{s}}), W(e_{\bar{t}})} (h_{W(e_{\bar{s}\restriction' \beta}), W(e_{\bar{s}})}(p_{\bar{s}\restriction' \beta}))=h_{W(e_{\bar{s}\restriction' \beta}), W(e_{\bar{t}})}(p_{\bar{s}\restriction' \beta}).
\end{equation*}

\end{proof}

\begin{claim}\label{3.11}
$C=\{h_{W(e_{\bar{s}}), W(e_{\bar{t}})}(p_{\bar{s}}): \bar{t}\in \Pi_{j<d} B_j, \bar{s}=\bar{t}\restriction a_{\alpha}\}$ consists of pairwise compatible conditions.
\end{claim}
\begin{proof}[Proof of the claim]
Let $\bar{t}_i\in \Pi_{j<d} B_j$, $\bar{s}_i \in \Pi_{j<d} T_j'(\alpha)$ with $\bar{s}_i\prec \bar{t}_i$, $i<2$. Consider $t^* = \{(t_0)_j: j<d\} \cap \{(t_1)_j: j<d\} , s^*=\{x\in \bigcup_{j<d} T'_j(\alpha): \text{for some } y\in t^*,\ x\prec y\}$. By \ref{CL3} in Lemma \ref{Cleanup}, we have $W(e_{\bar{t}_0})\cap W(e_{\bar{t}_1})=W(e_{t^*})$. Since by (\ref{ordertype}), for $i<2$, $(e_{\bar{t}_i}, e_{t^*}, <)\simeq (e_{\bar{s}_i}, e_{s^*}, <)$, we know $h_{W(e_{\bar{s}_i})), W(e_{\bar{t}_i})}(W(e_{s^*}))=W(e_{t^*})$ by \ref{CL4} in Lemma \ref{Cleanup}. Since $p_{\bar{s}_0}$ and $p_{\bar{s}_1}$ are compatible by \ref{3}, and in particular they agree on $W(e_{s^*})$, so $h_{W(e_{\bar{s}_0}), W(e_{\bar{t}_0})}(p_{\bar{s}_0})$ and $h_{W(e_{\bar{s}_1}), W(e_{\bar{t}_1})}(p_{\bar{s}_1})$ are compatible.
\end{proof}

Now we let $p^*$ be the greatest lower bound for $C$. Note that for each $j<d$, $k<\beta_j$, $p^*(e_{t^j_k},j)=t^j_k\restriction a_\alpha$.
To see this, let $\{s_0,\cdots, s_{d-1}\}=\bar{s}\in \Pi_{j<d} T_j'$, $\{t_0,\cdots, t_{d-1}\}=\bar{t}\in \Pi_{j<d} B_j$ be such that $t_j=t^j_k$ and $s_j=t^j_k\restriction a_\alpha$. As $p^*\leq h_{W(e_{\bar{s}}), W(e_{\bar{t}})}(p_{\bar{s}})$ and $p_{\bar{s}}(e_{s_j}, j)=s_j$ by \ref{2}, we know $h_{W(e_{\bar{s}}), W(e_{\bar{t}})}(p_{\bar{s}})(e_{t_j},j)=s_j$, hence $p^*(e_{t_j}, j)\succeq s_j$. Again by \ref{2} we know that $ht_{T_j} (p^*(e_{t_j}, j)) \leq a_{\alpha}$, so $p^*(e_{t_j}, j)=p^*(e_{t^j_k},j)= s_j$.

Further extend $p^*$ to $p'$ to ensure that 
\begin{equation}\label{maintaining}
\text{for each $j<d$ and $k<\beta_j$}, p'(e_{t^j_k},j)=t^j_k.
\end{equation} 

\begin{claim}\label{alpha}
There exist $p''\leq p'$ and $\{j_{\alpha, \bar{t}}<\delta_\alpha:  \bar{t}\in \Pi_{j<d} B_j\}$ such that for any $\bar{t}\in \Pi_{j<d} B_j \ p'' \restriction W(e_{\bar{t}})\Vdash  \dot{\tau}_{\alpha}(e_{\bar{t}})=j_{\alpha, \bar{t}}$.
\end{claim}

\begin{proof}[Proof of the claim]
Enumerate $\Pi_{j<d} B_j =\{\bar{t}_l: l<\epsilon\}$. We will build recursively a decreasing chain $\langle q_{l}: l<\epsilon\rangle$ such that for each $l<\epsilon$, $q_l \restriction W(e_{\bar{t}_l}) \Vdash \dot{\tau}_\alpha (e_{\bar{t}_l})= j_{\alpha, \bar{t}_l}$. At stage $l<\epsilon$, let $q^*$ be a lower bound for $\langle q_w: w<l\rangle$ and $j_{\alpha, \bar{t}_l}<\delta_\alpha$ such that $q^*\Vdash \dot{\tau}_\alpha (e_{\bar{t}_l})= j_{\alpha, \bar{t}_l}$. Since $\mathbb{P}\restriction W(e_{\bar{t}_l})$ contains a maximal antichain deciding the values of $\dot{\tau}_\alpha (e_{\bar{t}_l})$, there exists $q'\in \mathbb{P}\restriction W(e_{\bar{t}_l})$ and $q'\leq q^*\restriction W(e_{\bar{t}_l})$ such that $q'\Vdash \dot{\tau}_\alpha (e_{\bar{t}_l})=j_{\alpha,\bar{t}_l}$. Pick $q_l \leq  q', q^*$, then it would be as desired. Let $p''$ be a lower bound for $\langle q_l: l<\epsilon\rangle$, then it works.
\end{proof}

\begin{claim}\label{3.8}
$p''\Vdash \forall \beta < \alpha \; \forall \bar{t}\in \Pi_{j<d}B_j\ \dot{\tau}_{\beta} (e_{\bar{t}})=j_{\beta, \bar{t}\restriction a_{\beta+1}}$.
\end{claim}
\begin{proof}[Proof of the claim]
Let $\bar{x}=\bar{t}\restriction a_{\beta+1} \in \Pi_{j<d} T'_j(\beta+1)$.
Notice that $p''\leq h_{W(e_{\bar{x}}), W(e_{\bar{t}})}(p_{\bar{x}})$ and $p_{\bar{x}}\Vdash \dot{\tau}_\beta (e_{\bar{x}})=j_{\beta, \bar{x}}$. By \ref{CL2} in Lemma \ref{Cleanup}, 
we know that $h_{W(e_{\bar{x}}), W(e_{\bar{t}})}(p_{\bar{x}}) \Vdash \dot{\tau}_\beta(e_{\bar{t}})= j_{\beta, \bar{x}}$. The claim then follows.
\end{proof}

Now find $p'''\leq p''$ and $\chi<\kappa$ such that 
\begin{equation}\label{choice of chi}
\chi > \sup \{ ht_{T_j} (p''(i)(j)): i\in dom(p''), j<d\}
\end{equation} 
 and $p'''\Vdash \chi \in \bigcap \{ \dot{A}_{e_{\bar{t}}}^\beta: \beta\leq \alpha, \bar{t}\in \Pi_{j<d}B_j\}$. The reason why we can do this is $\dot{D}$ is forced to be $\kappa$-complete.
Extending $p'''$ if necessary we may assume that for any $l\in dom(p''')$ and  $j<d$, we have $ht_{T_j}(p'''(l)(j))\geq \chi$. 

Consider $C_j=\{p'''(e_{t})(j)\restriction \chi: t\in B_j\}$ for each $j<d$.
For any $\bar{t}=\{t_0, t_1,\cdots, t_{d-1}\}\in \Pi_{j<d} B_j$, 
\begin{equation*}
 p'''\Vdash f_\alpha (\dot{c}_{e_{t_0}}(0)\restriction \chi, \cdots, \dot{c}_{e_{t_{d-1}}}(d-1)\restriction \chi)=\dot{\tau}_\alpha(e_{\bar{t}})=j_{\alpha,\bar{t}},
\end{equation*} which implies
\begin{equation}\label{color1}
 f_{\alpha}(p'''(e_{t_0})(0)\restriction \chi, \cdots, p'''(e_{t_{d-1}})(d-1)\restriction\chi)=j_{\alpha, \bar{t}} 
\end{equation}
and similarly for any $\beta<\alpha$,  \begin{equation*}
 p'''\Vdash f_\beta (\dot{c}_{e_{t_0}}(0)\restriction \chi, \cdots, \dot{c}_{e_{t_{d-1}}}(d-1)\restriction \chi)=\dot{\tau}_\beta(e_{\bar{t}})=j_{\beta,\bar{t}\restriction a_{\beta+1}},
\end{equation*} 
which implies
\begin{equation}\label{color2}
 f_{\beta}(p'''(e_{t_0})(0)\restriction \chi , \cdots, p'''(e_{t_{d-1}})(d-1)\restriction\chi )=j_{\beta, \bar{t}\restriction a_{\beta+1}}.
\end{equation}

Note that for any $t\in B_j$, $p'''(e_t,j)\restriction \chi \succ p''(e_t,j) \succ p'(e_t,j)=t$ by (\ref{maintaining}). Let $C_j=T'_j(\alpha+1), j<d$ and let $a_{\alpha+1}=\chi$. For any $j<d$, and any $t\in B_j$, there exists a unique $t'\in C_j$, equivalently $t'\in T_j'(\alpha+1)$, such that $t'\succ t$. Thus this ensures the trees being constructed are strong subtrees.

Let $q\in \mathbb{P}$ be such that $dom(q)=dom(p''')$ and for any $k\in dom(q), j<d$, $q(k)(j)=p'''(k)(j)\restriction \chi$. Notice by the choice of $\chi$ as in (\ref{choice of chi}), $q\leq p''$.

For each $\bar{y}\in \Pi_{j<d} C_j$, let $\bar{t}=\bar{y}\restriction a_{\alpha}+1$ so that $\bar{t}\in \Pi_{j<d}B_j$ and define
\begin{enumerate}[label=\textbf{R.\arabic*}]
\item \label{Req1} $j_{\alpha, \bar{y}}:=j_{\alpha, \bar{t}}$
\item \label{Req2} $p_{\bar{y}}:=h_{W(e_{\bar{t}}), W(e_{\bar{y}})}(q\restriction W(e_{\bar{t}})).$

\end{enumerate}

Fix $\{y_0,\cdots, y_{d-1}\}=\bar{y}\in \Pi_{j<d} T_j'(\alpha+1)$ and let $\{t_0,\cdots, t_{d-1}\}=\bar{t}=\bar{y}\restriction a_{\alpha}+1$ so that $\bar{t}\in \Pi_{j<d}B_j$. We will verify the construction invariants are maintained. 

\begin{enumerate}

\item[\ref{1}] Since by Claim \ref{alpha}, $p'' \restriction W(e_{\bar{t}})\Vdash \dot{\tau}_\alpha (e_{\bar{t}})=j_{\alpha, \bar{t}}$ and $q\restriction W(e_{\bar{t}})\leq p''\restriction W(e_{\bar{t}})$, using \ref{CL2} in Lemma \ref{Cleanup} we know $p_{\bar{y}}\Vdash \dot{\tau}_\alpha (e_{\bar{y}})=j_{\alpha, \bar{t}}=j_{\alpha, \bar{y}}$. Similarly let $\bar{x}=\bar{t}\restriction a_{\beta+1}$ and for any $\beta<\alpha$, we have
\begin{equation*}
q \restriction W(e_{\bar{t}})\leq p''\restriction W(e_{\bar{t}})\leq h_{W(e_{\bar{x}}), W(e_{\bar{t}})}(p_{\bar{x}}).
\end{equation*} 
Since $p_{\bar{x}}\Vdash \dot{\tau}_\beta(e_{\bar{x}})=j_{\beta,\bar{x}}$ by \ref{1} in the induction hypothesis, we know $p_{\bar{y}}\Vdash \dot{\tau}_\beta (e_{\bar{y}})=j_{\beta,\bar{x}}$. By \ref{color2}, $f_\beta (\bar{y})=j_{\beta, \bar{x}}$.
This is as desired since $\bar{x}=\bar{y}\restriction' \beta+1$.

\item[\ref{2}] Observe that for any $k\in  dom(p_{\bar{y}}), j<d$, $ht_{T_j}(p_{\bar{y}}(k)(j))\leq \chi=a_{\alpha+1}$ and 
\begin{equation*}
 p_{\bar{y}}(e_{y_j},j)=h_{W(e_{\bar{t}}), W(e_{\bar{y}})}(q\restriction W(e_{\bar{t}}))(e_{y_j},j)=q(e_{t_j}, j)=p'''(e_{t_j})(j)\restriction\chi = y_j.
\end{equation*} 
  
\item[\ref{3}] It follows from the similar proof as in Claim \ref{3.11}. More precisely, given $\bar{y}_0, \bar{y}_1\in \Pi_{j<d} C_j = \Pi_{j<d} T'_j(\alpha+1)$, we need to show $p_{\bar{y}_0}$ and $p_{\bar{y}_1}$ are compatible. Let $\bar{t}_0 =\bar{y}_0\restriction a_{\alpha}+1, \bar{t}_1=\bar{y}_1\restriction a_{\alpha}+1$. By \ref{Req2}, $p_{\bar{y}_0}=h_{W(e_{\bar{t}_0}),W(e_{\bar{y}_0})}(q\restriction W(e_{\bar{t}_0}))$ and $p_{\bar{y}_1}=h_{W(e_{\bar{t}_1}),W(e_{\bar{y}_1})}(q\restriction W(e_{\bar{t}_1}))$. Let $y^*=\{(y_0)_j: j<d\}\cap \{(y_1)_j: j<d\}$ and $t^*=\{t\in (\bigcup \bar{t}_0 \cup \bigcup \bar{t}_1): \text{for some } h\in y^*, t\prec h \}$. We know that $(e_{\bar{y}_i}, e_{y^*},<)\simeq (e_{\bar{t}_i}, e_{t^*},<)$ for $i<2$. By \ref{CL3} in Lemma \ref{Cleanup}, we know $W(e_{\bar{y}_0})\cap W(e_{\bar{y}_1})=W(e_{y^*})$. By \ref{CL4} we get $h_{W(e_{\bar{t}_i}), W(e_{\bar{y}_i})}(W(e_{t^*}))=W(e_{y^*})$ for $i<2$. As $q\restriction W(e_{\bar{t}_0})$ agrees with $q\restriction W(e_{\bar{t}_1})$ on $W(e_{t^*})$, it follows that $p_{\bar{y}_0}$ agrees with $p_{\bar{y}_1}$ on $W(e_{y^*})$. But $dom(p_{\bar{y}_0})\cap dom(p_{\bar{y}_1})\subset W(e_{\bar{y}_0})\cap W(e_{\bar{y}_1})=W(e_{y^*})$. Hence these two conditions are compatible.

\item[\ref{4}] For any $\beta\leq \alpha$, let $\bar{s}=\bar{y}\restriction' \beta$ and recall the choice of $p^*$ following Claim \ref{3.11}, we have 

\begin{equation}
p^*\restriction W(e_{\bar{t}}) \leq  h_{W(e_{\bar{s}}), W(e_{\bar{t}})} (p_{\bar{s}}).
\end{equation}
Therefore, 
\begin{equation}
p_{\bar{y}}=h_{W(e_{\bar{t}}), W(e_{\bar{y}})}(q\restriction W(e_{\bar{t}})) \leq h_{W(e_{\bar{t}}), W(e_{\bar{y}})} (p^*\restriction W(e_{\bar{t}})) \leq h_{W(e_{\bar{s}}), W(e_{\bar{y}})}(p_{\bar{s}})
\end{equation}
as desired.
\end{enumerate}

\textbf{Stage $\alpha$ when $\alpha$ is a limit}.
The construction will be very similar to the successor case so we will just point out differences. Suppose we have constructed $\{T_j'(<\alpha): j<d\}$ and $\{a_i: i<\alpha\}$ satisfying the construction invariants. For $j<d$, let $B_j = \{t_k^j: k<\beta_j\}$ enumerate the nodes on $T_j(\bar{\alpha})$ where $\bar{\alpha}=\sup_{\gamma<\alpha} a_\gamma$ such that they are the upper bounds of branches through $T'_j(<\alpha)$.

We need the following claims, that are analogous to Claim \ref{3.10} and Claim \ref{3.11} respectively.
\begin{claim}\label{An1}
For each $\bar{t}\in \Pi_{j<d} B_j$, $\{h_{W(e_{\bar{t}\restriction a_\beta}), W(e_{\bar{t}})}(p_{\bar{t}\restriction a_\beta}): \beta<\alpha\}$ consists of pairwise compatible conditions, in fact it is decreasing as $\beta$ increases. 
\end{claim}
\begin{proof}[Proof of the claim]
Given $\beta_0< \beta_1<\alpha$, let $\bar{x}=\bar{t}\restriction a_{\beta_0}$ and $\bar{y}=\bar{t}\restriction a_{\beta_1}$. By \ref{4} of the construction invariants, $p_{\bar{y}}\leq h_{W(e_{\bar{x}}), W(e_{\bar{y}})}(p_{\bar{x}})$. Hence 
\begin{equation*}
h_{W(e_{\bar{y}}), W(e_{\bar{t}})}(p_{\bar{y}})\leq h_{W(e_{\bar{y}}), W(e_{\bar{t}})}(h_{W(e_{\bar{x}}), W(e_{\bar{y}})}(p_{\bar{x}}))=h_{W(e_{\bar{x}}), W(e_{\bar{t}})}(p_{\bar{x}}).
\end{equation*}

\end{proof}

For each $\bar{t}\in \Pi_{j<d} B_j$, let $q_{\bar{t}}$ denote the greatest lower bound for the set as in Claim \ref{An1}.

\begin{claim}
$\{q_{\bar{t}}: \bar{t}\in \Pi_{j<d} B_j\}$ consists of pairwise compatible conditions.
\end{claim}

\begin{proof}[Proof of the claim]
Suppose for the sake of contradiction that there exists $\bar{t}_0, \bar{t}_1 \in \Pi_{j<d} B_j$ such that $q_{\bar{t}_0} \bot q_{\bar{t}_1}$. Then there exists $\nu  \in dom(q_{\bar{t}_0}) \cap dom(q_{\bar{t}_1}), j<d$ such that $q_{\bar{t}_0}(\nu,j)$ and $q_{\bar{t}_1}(\nu,j)$ are incompatible in $(T_j, <_{T_j})$. Since for $i<2$, $q_{\bar{t}_i}(\nu,j) = \sup_{\beta<\alpha} h_{W(e_{\bar{t}_i \restriction a_\beta}), W(e_{\bar{t}_i})}(p_{\bar{t}_i\restriction a_\beta})(\nu, j)$, we can find some $\beta<\alpha$ large enough such that if we let $\bar{s}_0=\bar{t}_0 \restriction a_\beta, \bar{s}_1=\bar{t}_1 \restriction a_\beta$, both in $\Pi_{j<d} T'_j(\beta)$, then $h_{W(e_{\bar{s}_0}), W(e_{\bar{t}_0})}(p_{\bar{s}_0})(\nu,j)$ and $h_{W(e_{\bar{s}_1}), W(e_{\bar{t}_1})}(p_{\bar{s}_1})(\nu,j)$ are incompatible in $T_j$. Let $t^* = \bar{t}_0\cap \bar{t}_1$ then $W(e_{t^*})=W(e_{\bar{t}_0})\cap W(e_{\bar{t}_1})$. Let $s^*=\{z\in \bigcup_{j<d} T_j'(\beta): \text{for some }y\in t^*, z\prec y\}$. As in Claim \ref{3.11} we have $h_{W(e_{\bar{s}_i}), W(e_{\bar{t}_i})}(W(e_{s^*}))=W(e_{t^*})$. Note that $h_{W(e_{\bar{s}_0}), W(e_{\bar{t}_0})}$ and $h_{W(e_{\bar{s}_1}), W(e_{\bar{t}_1})}$ actually agree on $W(e_{s^*})$. Now let $\nu'\in W(e_{s^*})$ be such that $h_{W(e_{\bar{s}_i}), W(e_{\bar{t}_i})}(\nu')=\nu$. Then for $i<2$, 
\begin{equation*}
p_{\bar{s}_i}(\nu',j)=h_{W(e_{\bar{s}_i}), W(e_{\bar{t}_i})}(p_{\bar{s}_i})(\nu,j).
\end{equation*} 
By the assumption we have that $p_{\bar{s}_0}(\nu',j)$ is incompatible with $p_{\bar{s}_1}(\nu',j)$, but this contradicts with our induction hypothesis on the construction invariants.
\end{proof}

The rest of the proof is almost identical with the successor case except that we do not need to determine the corresponding values for $\dot{\tau}_\alpha$ now. More precisely, we follow the proof after Claim \ref{3.11} but skip Claim \ref{alpha}.

\end{proof}

\begin{remark}\label{WC}
If inaccessible $\kappa$ satisfies $HL^{tc}(1,2,\kappa)$, then $\kappa$ is weakly compact. To see this, given $g: [\kappa]^2\to 2$, we have $g_\alpha: \kappa \to 2$ for each $\alpha<\kappa$ such that $g_{\alpha}(\beta)=g(\alpha,\beta)$. Fix a bijection $h$ between $\kappa$ and $2^{<\kappa}$. We may view $g_\alpha$ as colorings of $2^{<\kappa}$ via $h$. By tail cone homogeneity, we have a strong subtree $T'\subset 2^{<\kappa}$ that is tail cone homogeneous with respect to $\langle g_\alpha: \alpha<\kappa\rangle$. Pick a branch through $T'$, then this is a subset of $2^{<\kappa}$ of size $\kappa$ such that any $g_{\alpha}$ is constant except for a subset of size $<\kappa$. But this clearly implies that $\kappa$ is weakly compact.
\end{remark}

The following is a straightforward recursive construction.

\begin{lem}\label{trimming}
For any nice tree $T$, any strong subtree $T'\subset T$ with witnessing level set $A'\in [\kappa]^\kappa$ and any $A\subset A'$ with $|A|=\kappa$, there exists a strong subtree $T^*\subset T'$ such that $T^*$ is a strong subtree of $T$ with witnessing level set $A$.
\end{lem}

\begin{cor}\label{tailconeapp}
Let $0<i<d\in \omega$, $\delta<\kappa$ be given and assume $HL^{tc}(i,\delta,\kappa)$. For any sequence of nice trees $\langle T_j: j<d\rangle$, $f: \Pi_{j<d} T_j\to \delta$ and $B\subset d$ with $|d\backslash B|=i$, there exist strong subtrees $\langle T'_j\subset T_j: j<d\rangle$ sharing the same witnessing level subset satisfying the following: 

For any $\bar{t}=\langle t_j: j\in B \rangle\in \Pi_{j\in B} T'_j$ with $\xi=\max \{ht_{T_j'}(t_j): j\in B\}$ and any level sequence $\bar{s}\in \Pi_{j\in d-B} T'_j$ with height greater than $\xi$, $f(\bar{s}\cup \bar{t})=f((\bar{s}\restriction' \xi+1) \cup \bar{t})$. Note that here we do not require $\bar{t}$ to be a level sequence.
\end{cor}

\begin{proof}
Enumerate $\Pi_{j\in B} T_j$ as $\{\bar{t}^l\in \Pi_{j\in B} T_j : l<\kappa\}$. Define $f_l: \Pi_{j\in d-B} T_j\to \delta$ such that $f_l(\bar{s})=f(\bar{t}^l \cup \bar{s})$. Applying $HL^{tc}(i,\delta,\kappa)$, we get a sequence of strong subtrees $\langle T^*_j: j\in d-B\rangle$ sharing the same level subset $A'\in [\kappa]^\kappa$ with increasing enumeration $\langle a'_k: k<\kappa \rangle$ so that for any $k<\kappa$, any level sequence $\bar{s}\in \Pi_{j\in d-B} T^*_j$ with $(\Pi_{j\in d-B} T_j)$-height greater than $a'_k$, $f_k (\bar{s})=f_k(\bar{s}\restriction^* k+1)=f_k(\bar{s}\restriction a'_{k+1})$. Thin $A'$ to $A\subset A'$ with increasing enumeration $\langle a_k: k<\kappa\rangle$ such that for any $\bar{t}^i\in \Pi_{j\in B} T_j (\leq a_k)$, any level sequence $\bar{s}\in \Pi_{j\in d-B} T_j^*$ with $(\Pi_{j\in d-B} T_j)$-height greater or equal to $a_{k+1}$, $f(\bar{t}^i\cup \bar{s})=f_i(\bar{s})=f_i(\bar{s}\restriction a_{k+1})=f(\bar{t}^i \cup (\bar{s}\restriction a_{k+1}))$.
 
Apply Lemma \ref{trimming} to get strong subtrees $T'_j\subset T_j$ for $j<d$ with $A$ being the common witnessing level set and whenever $j\in d-B$, $T'_j\subset T_j^*$. It is now easy to verify that these trees are as desired.
\end{proof}

\begin{thm}\label{DimInduct}
Let $d\in \omega$ and assume $HL^{tc}(d,\delta,\kappa)$. Then $HL(d+1,\delta,\kappa)$ holds.
\end{thm}

\begin{proof}
The hypothesis implies that $\kappa$ is weakly compact by Remark \ref{WC}. By Theorem \ref{EquiForm} it suffices to show $SDHL'(d+1,\delta,\kappa)$.
Given nice trees $\langle T_j: j\leq d\rangle$ and a coloring $f: \Pi_{j\leq d} T_j \to \delta$, apply Corollary \ref{tailconeapp}, we can find strong subtrees $T_j'$ sharing the same level set such that for any $x\in T_0'$ with $T_0'$-height $\xi<\kappa$ and any level sequence $\bar{y}\in \Pi_{1\leq j\leq d} T_j'(\zeta)$ for some $\zeta>\xi$, then $f(x,\bar{y})=f(x,\bar{y}\restriction' \xi+1)$. Our goal is find a somewhere dense monochromatic matrix for $\Pi_{j\leq d} T_j'$.

Apply Lemma \ref{trimming} to find a sequence of strong subtrees $\langle T_j^* \subset T'_j: j<d \rangle$ with witnessing level set $\{\xi+1: \xi<\kappa\}$ for $1\leq j\leq d$. Notice that for any $\xi<\kappa,  T_j^*(\xi) \subset T_j'(\xi+1)$ and for any $ r\in T_j^*, Succ_{T_j^*} (r) = Succ_{T_j'}(r)$.

For each $x\in [T_0']$, consider the following induced level coloring $f_x: \Pi_{1\leq j\leq d} T_j^* \to \delta$ defined as follows: for any $\chi<\kappa, \bar{y}\in \Pi_{1\leq j\leq d} T^*_j(\chi)$, $f_x (\bar{y})=f(x(\chi), \bar{y})$. Recall $HL^{tc}(d,\delta,\kappa)$ implies $\kappa$ is weakly compact and $HL(d,\delta,\kappa)$ which in turn implies $DSHL(d,\delta,\kappa)$ by Theorem \ref{EquiForm}. Apply $DSHL(d, \delta,\kappa)$ to $\Pi_{1\leq j\leq d} T_j^*$ and $f_x$. We know there exists a level sequence $\bar{t}_x\in \Pi_{1\leq j\leq d} T_j^*$ of height $\nu$ and $\gamma_x<\delta$ such that for any $\nu'\in \kappa\backslash (\nu+1)$ there exists a $\bar{t}_x$-$\nu'$-dense level matrix $\bar{M}$ with $f_x'' \bar{M}=\{\gamma_x\}$. Note that any density notion here is with respect to $\Pi_{1\leq j \leq d} T_j^*$.

Now consider the map $x\in [T_0']\mapsto (\bar{t}_x, \gamma_x)$. This is a partition of $[T_0']$ into $\kappa$ many pieces. Therefore, there exists $(\bar{t},\gamma)$ and $s\in T_0'$ such that $\exists^\tau x\in N^{T_0'}_{s} \ (\bar{t}_x, \gamma_x)=(\bar{t},\gamma)$. Recall Definition \ref{density} for the meaning of $\exists^\tau x\in N^{T_0'}_{s}$. Extending if necessary we might assume $ht_{T_0'}(s)=\beta+1$ and $ht_{\Pi_{1\leq j\leq d} T_j^*} (\bar{t})=\beta$. Let $\bar{t}=\langle t_j:1\leq  j\leq d\rangle$. Hence we have the following:

For each $x\in T_0'[s], \nu<\kappa$, there exist $\nu'<\kappa, \nu' \geq ht_{T_0'}(x)$, $x'\in T_0'(\nu')[x]$ and level matrix $\bar{M}=M_1\times \cdots \times M_d \subset \Pi_{1\leq j\leq d} T_j^*(\nu')[t_j]$ that is $\bar{t}$-$\nu$-dense such that $f'' x' \times \bar{M} =\{\gamma\}$.

Now enumerate the immediate successors in $T_0'$ of $s$ as $\{s_k: k\leq \mu \}$. Note $s_\mu$ is the last element in this enumeration. Recursively we build level matrices with increasing heights $\{\bar{M}_k \subset \Pi_{1\leq j\leq d} T^*_j: k\leq \mu\}$ and $\{s_k'\in T'_0[s_k]: k\leq \mu \}$ such that 
\begin{enumerate}
\item $\bar{M}_0$ is $\bar{t}$-$(\beta+1)$-dense in $\Pi_{1\leq j\leq d} T^*_j$ and $\forall k\leq \mu $, $\bar{M}_k$ is $\bar{t}$-$\bar{\alpha}$-dense in $\Pi_{1\leq j \leq d} T_j^*$ where $\bar{\alpha}=\sup_{l<k} ht_{\Pi_{1\leq j \leq d} T_j^*} (\bar{M}_l)$;
\item for all $k\leq \mu $, $ht_{T_0'}(s_k')=ht_{\Pi_{1\leq j\leq d} T^*_j}(\bar{M}_k)$;
\item for all $k\leq \mu$, $f(\{s_k'\}\times \bar{M}_k)=\{\gamma\}$.
\end{enumerate}

Pick level matrix $\bar{N}\subset \bar{M}_\mu $ such that for each $1\leq j\leq d$, $N_j$ is $t_j$-$(\beta +1)$-dense in $T_j^*$ and $\bar{x}\in \bar{N}$ implies that for all $k<\beta \ \exists \lambda \  \bar{x}\restriction^* \lambda \in \bar{M}_k$.

We now check that $\{s_k': k\leq \mu\}\times \bar{N}$ is $(s, \bar{t})$-$(\beta+2)$-dense matrix in $\Pi_{j\leq d} T_j'$ of color $\gamma$. The density is clear since $T_j^*(\beta+1)[t_j]=T_j'(\beta+2)[t_j], 1\leq j \leq d$. For any $k\leq \beta$ and $\bar{y}\in \bar{N}$, we need to show $f(s'_k,\bar{y})=\gamma$. If $k=\mu$ we are done. If $k<\mu $, by the choice of $\bar{N}$, there exists $\lambda<\kappa$ such that $\bar{x}=\bar{y}\restriction^* \lambda \in \bar{M}_k$. Note that $f(s'_k, \bar{x})=\gamma$.
We have $ht_{T'_0}(s'_k)=ht_{\Pi_{1\leq j\leq d} T^*_j}(\bar{x})=\lambda$ so $ht_{\Pi_{1\leq j\leq d} T'_j}(\bar{x})=\lambda+1$. Since $\bar{y}$ is also a level sequence in $\Pi_{1\leq j\leq d}T_j'$, we know $f(s_k', \bar{y})= f(s_k', \bar{y}\restriction' \lambda+1)=f(s_k', \bar{x})=\gamma$.
\end{proof}

It is a reasonable question to ask if the tail cone version of the Halpern-Läuchli theorem is true at all assuming $\kappa$ is some large enough cardinal.
In particular: 
\begin{question}
Is $HL^{tc}(1,\delta,\kappa), \delta<\kappa$ true in ZFC for some large enough cardinal $\kappa$? 
\end{question}

A positive answer will enable us to establish results of higher dimension via the inductive argument as in Theorem \ref{DimInduct}.

\section{A polarized partition relation for large saturated linear orders}\label{polarize}

We show that tail cone homogeneity can be used to prove a higher analog of a theorem by Laver \cite{MR754925} (the 2-dimensional case is due to Galvin) regarding a polarized partition relation for dense linear orders.


\begin{notation}
We use $\forall^{<\kappa}$ to mean ``for all except for $<\kappa$ many''. For example, let $\varphi(x)$ be a formula with a free variable $x$, then $\forall^{<\kappa} x \ \varphi(x)$ means $|\{x: \neg\varphi(x)\}|<\kappa$.
\end{notation}

\begin{lem}\label{AlmostAll}
Let $d\in \omega, d\geq 2$, $\langle T_j: j<d\rangle$ nice trees and a coloring $f: \Pi_{j<d}T_j\to \delta$ for some $\delta<\kappa$ be given. Assume $HL^{tc}(d-1,\delta,\kappa)$. There exist strong subtrees $T_j^*\subset T_j$ not necessarily sharing the same level set with respect to $T_j, j<d$, such that there exists $\gamma\in \delta$, $(\forall^{<\kappa} x_0\in T^*_0) \cdots (\forall^{<\kappa} x_{d-1}\in T^*_{d-1}) \ f(x_0,\cdots, x_{d-1})=\gamma$.

\end{lem}

\begin{proof}

We first shrink the trees to satisfy the following claim.

\begin{claim}
There exist strong subtrees $T_j'\subset T_j, j<d$ not necessarily sharing the same level set such that for any $x\in T_0'$ with $ht_{T_0'}(x)=\xi$ and any $y_j\in T_j'$ whose heights are increasing with $ht_{T_1'}(y_1)> \xi$, 
\begin{equation*}
f(x,y_1,\cdots , y_{d-1})=f(x, y_1\restriction' \xi, \cdots, y_{d-1}\restriction' \xi).
\end{equation*} 

\end{claim}
\begin{proof}[Proof of the claim]
As the restrictions taken in the following proof are with respect to different subtrees, we remind the reader of Notation \ref{restriction}.
We shrink the trees in $d$ stages. For each $j<d$, let $T^0_j=T_j$. Assume at stage $0\leq i < d-1$, we already have strong subtrees $\{T^i_j: j<d\}$ such that if $\bar{x}\in \Pi_{j<d} T^i_j$ is such that its respective heights are increasing with $\xi=ht_{T^i_{d-1-i}}(x_{d-1-i})$, then \begin{equation}\label{IH}
\begin{multlined}
f(x_0,\cdots, x_{d-1-i}, x_{d-i},\cdots,  x_{d-1})= f(x_0, \cdots, x_{d-1-i}, x_{d-i}\restriction^i \xi, \cdots, x_{d-1}\restriction^i \xi).
\end{multlined}
\end{equation}

Stage $i+1: $ Apply Corollary \ref{tailconeapp} and find strong subtrees $U_j\subset T^i_j$ sharing the same level set $A$ with respect to $\Pi_{j<d} T^i_j$ such that for any $\bar{x}\in \Pi_{j<d-1-i} U_j$ with increasing heights and $\xi=ht_{U_j}(x_{d-2-i})$ and any $\chi>\xi, \bar{y}\in \Pi_{d-1-i\leq j<d} U_{j}(\chi)$, $f(\bar{x}^\frown \bar{y})=f(\bar{x}^\frown (\bar{y}\restriction_{\Pi_{d-1-i\leq j<d}U_j} \xi+1))$.
Now let $T^{i+1}_j = U_j$ for $j<d-1-i$ and for $d-1-i\leq j<d$ apply Lemma \ref{trimming} to get strong subtrees $T^{i+1}_j$ of $ U_j$ with witnessing level set $\{\xi+1: \xi<\kappa\}$, hence $T^{i+1}_j(\xi)\subset U_j(\xi+1)$ for all $\xi<\kappa$. We need to verify that for any
$\bar{x}\in \Pi_{j<d} T^{i+1}_j$ whose respective heights are increasing with $\xi=ht_{T^{i+1}_{d-2-i}}(x_{d-2-i})$, then
 \begin{equation*}
 \begin{multlined}
f(x_0,\cdots, x_{d-2-i}, x_{d-1-i},\cdots,  x_{d-1})=  \\ f(x_0, \cdots, x_{d-2-i}, x_{d-1-i}\restriction^{i+1} \xi, \cdots, x_{d-1}\restriction^{i+1} \xi).
 \end{multlined}
\end{equation*}
Let $\xi'=ht_{T^{i+1}_{d-1-i}}(x_{d-1-i})$.
Note that 
\begin{equation*}
\begin{multlined}
f(x_0,\cdots, x_{d-2-i}, x_{d-1-i},\cdots,  x_{d-1})= \\ f(x_0, \cdots, x_{d-2-i}, x_{d-1-i},x_{d-i} \restriction^{i+1} \xi', \cdots, x_{d-1}\restriction^{i+1}\xi').
\end{multlined}
\end{equation*}
The reason is that $\langle T^{i+1}_j: d-1-i\leq j<d\rangle$ share the same level set with respect to $\langle T^i_j: d-1-i\leq j<d \rangle$ so it follows from the induction hypothesis (\ref{IH}).
But now it follows that 
\begin{equation*}
\begin{split}
f(x_0, \cdots, x_{d-2-i}, x_{d-1-i},x_{d-i} \restriction^{i+1} \xi', \cdots, x_{d-1}\restriction^{i+1} \xi')=\\
f(x_0, \cdots, x_{d-2-i}, x_{d-1-i},x_{d-i} \restriction_{U_{d-i}} \xi'+1, \cdots, x_{d-1}\restriction_{U_{d-1}}\xi'+1)=\\
f(x_0, \cdots, x_{d-2-i}, x_{d-1-i}\restriction_{U_{d-1-i}} \xi+1,x_{d-i} \restriction_{U_{d-i}} \xi+1, \cdots, x_{d-1}\restriction_{U_{d-1}}\xi+1)=\\
f(x_0, \cdots, x_{d-2-i}, x_{d-1-i}\restriction^{i+1} \xi, \cdots, x_{d-1}\restriction^{i+1} \xi),
\end{split}
\end{equation*}
since for all $d-1-i\leq j <d, \xi<\kappa$, $T^{i+1}_j (\xi) \subset U_j(\xi+1)$. This finishes the inductive construction. Finally let $T_j'=T_j^{d-1}, j<d$ and it is straightforward to see that they are as desired.

\end{proof}

Now consider the coloring $f$ restricted to $\Pi_{j<d}T_j'$. Apply $HL(d,\delta,\kappa)$, which we have by Theorem \ref{DimInduct}. We can find strong subtrees $T_j^*$ of $T_j'$ for all $j<d$ sharing the same witnessing level set with respect to $\langle T_j': j<d\rangle$ such that $f'' \bigcup_{\xi<\kappa}\Pi_{j<d} T^*_j(\xi)=\{\gamma\}$ for some $\gamma<\delta$. To see this is desired, given $\bar{x}\in \Pi_{j<d} T^*_j$ whose respective heights are increasing and $\xi=ht_{T_0^*}(x_0)$, we see that 
\begin{equation*}
f(x_0,x_1,\cdots, x_{d-1})=f(x_0, x_1\restriction^* \xi,\cdots, x_{d-1}\restriction^* \xi)=\gamma
\end{equation*}
 as desired.

\end{proof}

\begin{thm}\label{mainproof}
Let $d\in \omega$ and suppose $HL^{tc}(d,<\kappa,\kappa)$. Then \begin{equation}\label{polarizerelation}
\begin{pmatrix}
\eta_\kappa \\
\vdots \\
\eta_\kappa
\end{pmatrix} \to \begin{pmatrix}
\eta_\kappa \\
\vdots \\
\eta_\kappa
\end{pmatrix}^{\underbrace{1,\cdots, 1}_{d+1}}_{<\kappa, (d+1)!}
\end{equation}
\end{thm}

\begin{proof}
Given $f: \Pi_{j<d+1}2^{<\kappa} \to \delta$ for some $\delta<\kappa$, our aim is to find for $j<d+1$, $T_j\subset 2^{<\kappa}$ such that $(T_j, <_{lex})$ contains a copy of a $\kappa$-saturated linear order and $|f'' \Pi_{j<d+1}T_j |\leq (d+1)!$.

\begin{claim}
There exists strong subtrees $T_j'\subset 2^{<\kappa}, j<d+1$, not necessarily sharing the same witnessing level set with respect to $2^{<\kappa}$, such that for any permutation $\pi: d+1\to d+1$, there exists $\gamma_\pi<\delta$ such that
\begin{equation}\label{Cleaning}
(\forall^{<\kappa} t_0\in T_{\pi(0)}') \cdots (\forall^{<\kappa} t_d\in T_{\pi(d)}') \ f(t_{\pi^{-1}(0)},\cdots,t_{\pi^{-1}(d)})=\gamma_\pi.
\end{equation}
 Furthermore, by possibly throwing away $<\kappa$ many nodes from each tree, we can assume the first quantification of Equation \ref{Cleaning} is actually $\forall t_0\in T_{\pi(0)}'$
\end{claim}
\begin{proof}[Proof of the claim]
Enumerate all permutations of $d+1$ and iteratively apply Lemma \ref{AlmostAll} finitely many times.
\end{proof}

Given the claim with $\langle T_j': j<d+1\rangle, \{\gamma_\pi: \pi \in Perm(d+1)\}$ where $Perm(S)$ is the collection of permutations of set $S$, we can recursively build nice subtrees $\langle T_j: j<d+1\rangle$. Recall the definition of a nice tree as in Definition \ref{tree}. Note that we do not require $T_j'$ to be a strong subtree of $T_j$ for $j<d+1$.

We will construct these trees in $\kappa$ stages so that each stage we only pick $<\kappa$ many elements. For $j< d+1$ and each $\beta<\kappa$, let $\{s_\beta^{j,k}: k<i_{\beta,j}\}$ be the set of nodes picked in $T_j'$ at stage $(d+1)\beta+j$. 
Suppose for some $p\leq d$ and some ordinal $\alpha<\kappa$ the construction reaches stage $(d+1)\alpha+p$. Let $A^j_{\alpha,p}=\{ (s^{j,k}_{\beta})_{k<i_{\beta,j}}: \beta<\alpha\} $ if $ j\geq p$ and $A^j_{\alpha,p}=\{ (s^{j,k}_{\beta})_{k<i_{\beta,j}}: \beta\leq \alpha\}$ if $ j<p$. Note that $i_{\beta, j}<\kappa$ if defined.

Define a partial binary relation $\sqsubset$ on $\bigcup_{j<d} T_j$ such that $s\sqsubset s'$ iff there exist $\beta, \beta'<\kappa, j \neq j' < d, k<i_{\beta,j}, k'<i_{\beta',j'} $ such that $s=s^{j,k}_\beta, s'=s^{j',k'}_{\beta'}$ and $\beta<\beta'$ or ($\beta=\beta'$ and $j<j'$). Notice that if $s\sqsubset s'$, then $s$ and $s'$ must have been picked during the construction. Intuitively one element precedes the other if it appears earlier in the construction. Further suppose during the construction we ensure 
\begin{enumerate}
\item  $ (\forall j<d) (\forall \beta<\alpha) (\forall k<i_{\beta,j}) \ s^{j,k}_\beta$ has $T_j'$-height greater than $\beta$ and $(\forall j<p) (\forall k<i_{\alpha,j})\ s^{j,k}_\alpha$ has $T_j'$-height greater than $\alpha$;
\item \label{Second} 
For 
\begin{itemize}
\item any $l\leq d$
\item any permutation $\langle j_0, \cdots, j_l\rangle$ of a subset of $d+1$ of size $l+1$
\item any $t_0 \in T'_{j_0}, \cdots, t_l\in T'_{j_l}$ satisfying that for $t_0\sqsubset t_1\sqsubset \cdots \sqsubset t_l$, and any $\pi\in Perm(d+1)$ with $\pi(w)=j_w$ for $w\leq l$ (in other words, $\pi$ extends the partial permutation $\langle j_0, \cdots, j_l\rangle$)
\end{itemize}
 we have
\begin{equation}\label{Forbidden}
(\forall^{<\kappa} t_{l+1} \in T'_{{\pi(l+1)}}) \cdots (\forall^{<\kappa} t_{d}\in T'_{{\pi(d)}}) f(t_{\pi^{-1}(0)},\cdots, t_{\pi^{-1}(l)}, t_{\pi^{-1}(l+1)},\cdots, t_{\pi^{-1}(d)})=\gamma_\pi.
\end{equation}
In particular this implies 
\begin{equation*}
(\forall \bar{x}\in \Pi_{j<d}A^j_{\alpha,p}) \ f(\bar{x})\in \{\gamma_\pi: \pi\in Perm(d+1)\}.
\end{equation*}
\end{enumerate}
At stage $(d+1)\alpha + p$: Enumerate the terminal nodes in $A^p_{\alpha,p}$ or top nodes for branches through $A^p_{\alpha,p}$ as $\{w_j: j<\chi\}$. For each $w_j$ we aim to add incompatible $w^0_j, w^1_j$ extending it of height $>\alpha$ satisfying the 2 requirements above. This is possible since for $l<d$ and for each $t_0, \cdots, t_l \in \bigcup_{j<d, j\neq p} A^j_{\alpha,p}$ satisfying that $t_0\sqsubset t_1\sqsubset \cdots \sqsubset t_l$ and no two nodes belong to the same tree, only $<\kappa$ many nodes are forbidden. Hence in total only $<\kappa$ many nodes are forbidden by (\ref{Forbidden}) and by the remark following (\ref{Cleaning}). Hence we can find such $w_j^0, w_j^1$ above $w_j$ of heights greater than $\alpha$ maintaining (\ref{Second}). $\{w_j^{k}: j<\chi, k<2\}$ will be picked from $T_j'$. In other words, $\{s_{\alpha}^{p,k}: k<i_{\alpha,p}\}= \{w_j^{k}: j<\chi, k<2\}$.

For each $j<d+1$, consider $T_j=\bigcup_{\alpha<\kappa} A^j_{\alpha,0}$. Then by construction each $T_j$ is rooted, perfect, and every maximal branch through $T_j$ is of order type $\kappa$. Thus each $T_j$ is nice. It is easy to see that for any $j<d+1$, $(T_j, <_{lex})$ contains a copy of a $\kappa$-saturated linear order. But now we are done as $f'' \Pi_{j<d+1} T_j \subset\{\gamma_\pi: \pi\in Perm(d+1)\}$.

\end{proof}

\begin{remark}
$(d+1)!$ is the best possible in the theorem above. To see $(d+1)!-1$ is not enough, simply define a coloring on $(2^{<\kappa}, <_{lex})^{d+1}$ with $Perm(d+1)$ being the color set such that whenever $\bar{x}\in (2^{<\kappa}, <_{lex})^{d+1}$ is such that the respective heights are distinct, then it gets color $\pi$ where $\pi\in Perm(d+1)$ satisfies that $ht (x_{\pi(0)})<  \cdots < ht(x_{\pi(d)})$.
\end{remark}

\begin{remark}
For $\triangleleft$ a well ordering of $\eta_\kappa$ of order type $\kappa$ such that any $(x_0,\cdots, x_d) \in \eta_\kappa ^{d+1}$ is associated with a unique type $\pi\in Perm(d+1)$ where $x_{\pi(0)}\triangleleft \cdots \triangleleft x_{\pi(d)}$, then essentially the same proof shows that there exist $X_j\subset \eta_\kappa$ of order type $\eta_\kappa$ for $j\leq d$ such that tuples from $\Pi_{j<d+1} X_j$ of the same type get the same color. 

\end{remark}

\begin{remark}
$
\begin{pmatrix}
\eta_\kappa \\
\eta_\kappa
\end{pmatrix} \to \begin{pmatrix}
\eta_\kappa \\
\eta_\kappa
\end{pmatrix}^{2}_{<\kappa, 2}
$ implies 
$
\begin{pmatrix}
\kappa \\
\kappa
\end{pmatrix} \to \begin{pmatrix}
\kappa \\
\kappa
\end{pmatrix}^{2}_{\omega, 2}.
$
To see this, let $h$ be a bijection between $\kappa$ and $2^{<\kappa}$. Given $f: \kappa\times \kappa\to \omega$, we use $h$ to transfer $f$ to a coloring $g: 2^{<\kappa}\times 2^{<\kappa}\to \omega$. The assumption implies that there exists $T_0, T_1$ that are nice subtrees of $2^{<\kappa}$ such that $|g'' T_0\times T_1|\leq 2$. Now letting $A_0= h^{-1} T_0, A_1 = h^{-1} T_1$ we have $|h'' A_0\times A_1|\leq 2$.
The latter by a Theorem of Todorcevic implies $\neg \square(\kappa)$ (see for example Theorem 6.3.2 in \cite{MR2355670}), which in turn implies $\kappa$ is weakly compact in $L$ by a theorem of Jensen. Note that $HL^{tc}(1,2,\kappa)$ implies $\kappa$ is weakly compact (see Remark \ref{WC}), which in turn implies $\kappa$ is weakly compact in $L$ by downward absoluteness.
\end{remark}

Theorem \ref{main} easily follows from Theorem \ref{Tail cone homogeneity in full generality} and Theorem \ref{mainproof}. One might ask: 
\begin{question}
Is the conclusion of Theorem \ref{mainproof} a consequence of some large cardinal hypothesis? 
\end{question}

Very recently, Dobrinen and Hathaway \cite{DobrinenHathaway2} showed that if $HL^{tc}(d,<\kappa,\kappa)$ holds in the ground model, then after forcing with a poset of size $<\kappa$, $HL^{tc} (d,<\kappa,\kappa)$ remains true.

It was pointed out in \cite{MR2520110} that Theorem \ref{unpo} is not a consequence of any large cardinal hypothesis. More precisely, by a theorem of Hajnal and Komj\'ath \cite{MR1606036}, there exists a forcing of size $\aleph_1$ that adds a linear order $\theta$ of size $\aleph_1$, that is strongly non-Ramsey, in the sense that for any linear ordering $\varphi$, $\varphi \not \to [\theta]^2_{\omega_1}$, namely for any $\omega_1$-coloring of $[\varphi]^2$, any copy of $\theta$ in $\varphi$ gets all colors.

Combining these two observations, it is consistent that 
$\eta_\kappa \not \to (\eta_\kappa)^2_{<\kappa, <\omega_1}$ while $(\ref{polarizerelation})$ holds for all $d\in \omega$.


\section{The Halpern-Läuchli theorem at $\kappa$ does not imply $\kappa$ is weakly compact}\label{NonImp}

In this section, we show that relative to the existence of a measurable cardinal it is possible that $\kappa$ is a strongly inaccessible cardinal but not weakly compact yet $HL(1,\delta,\kappa)$ holds for all $\delta<\kappa$. At the cost of stronger hypothesis, we show it is consistent that for all $d\in \omega$ and $\delta<\kappa,  HL(d,\delta,\kappa)$ holds while $\kappa$ is strongly inaccessible but not weakly compact, in contrast with the situation in the tail cone version of the Halpern-Läuchli theorem and in Theorem 2.5 of \cite{MR2520110} (regarding coloring unordered $m$-sized antichains of a single tree for finite $m$) where any inaccessible $\kappa$ satisfying those theorems must be weakly compact.

Recall the following definitions.

\begin{definition}
Let $\kappa, \lambda, \eta$ be cardinals. $I\subset P(\kappa)$ an ideal on $\kappa$ is \begin{itemize}
\item non-trivial if $\kappa\not \in I$;
\item $\lambda$-complete if for any $\alpha<\lambda$ and $X_i\in I, i<\delta$, $\bigcup_{i<\alpha} X_i\in I$;
\item $\eta$-saturated if $P(\kappa)/I$ has $\eta$-c.c, namely, any collection $X\subset P(\kappa)$ such that for any $A,B\in X$, $A,B\not \in I$ and $A\cap B\in I$, in other words, $A$ and $B$ are almost disjoint modulo $I$, satisfies that $|X|<\eta$. 
\end{itemize}
\end{definition}

For the rest of the section, fix an inaccessible cardinal $\kappa$ and a $\kappa$-complete $\kappa$-saturated ideal $I$ on $\kappa$.

We list some standard facts, which can be found in \cite{MR2768692}. 

\begin{fact}\label{IdealFacts}
Let $G$ be a generic ultrafilter on $\mathbb{P}=P(\kappa)/I$ over $V$ then 
\begin{enumerate}

\item $I$ is precipitous, namely, in $V[G]$, the ultrapower $Ult(V,G)=\{[f]_G: f\in V\}$ is well-founded.
Let $j: V\to M\simeq Ult(V,U)$ be the ultrapower map in $V[G]$ and $M$ is a transitive class.
\item $V[G]\models M^\kappa \subset M$.\label{point}

\item $G$ is $V\textendash \kappa$-complete, meaning for any $\alpha<\kappa$ and any $\langle A_i : i<\alpha\rangle \in V$ with $A_i\in G$ for all $i<\alpha$, $\bigcap_{i<\alpha} A_i \in G$.
\end{enumerate}
\end{fact}

\begin{thm}
Suppose $\kappa$ is an inaccessible cardinal that carries a $\kappa$-saturated $\kappa$-complete ideal, then for any $ \delta<\kappa, \  HL(1,\delta,\kappa)$.
\end{thm}

\begin{proof}
Let $T$, a nice tree and $f: T\to \delta$ be given.
Without loss of generality, we might assume $T$ is a subtree of $\kappa^{<\kappa}$.
For each $q\in T, \gamma\in \delta$, define $A_{q,\gamma}$ as in Theorem \ref{WeakCompactness}.
Let $G$ be $\mathbb{P}=P(\kappa)/I$-generic over $V$. In $V[G]$, let $j: V\to M\simeq Ult(V,G)$ be the generic ultrapower embedding.

\begin{claim}\label{NoNewBoundedSubset}
$(\kappa^{<\kappa})^M=(\kappa^{<\kappa})^{V[G]}=(\kappa^{<\kappa})^V$ and $\kappa$ remains strongly inaccessible in both $V[G]$ and $M$.
\end{claim}
\begin{proof}[Proof of the claim]
The first equality follows from Fact \ref{IdealFacts}. To show the second equality, first note that $\kappa$ remains regular in $V[G]$ hence in $M$ since forcing with $\mathbb{P}=P(\kappa)/I$ is $\kappa$-c.c. Given any $x\in \kappa^{<\kappa}$ in $V[G]$, by closure we know that $x\in M$. By the regularity of $\kappa$ in $M$, we can find sufficiently large $\theta \in \kappa $ and $\alpha<\kappa$ such that $x\in (\theta^{\alpha})^M$. However since $j$ is elementary we have $j((\theta^\alpha)^V)= (\theta^\alpha)^M=(\theta^\alpha)^V$ since $(\theta^\alpha)^V\in H(\kappa)^V$ and $j\restriction (H(\kappa))^V = id \restriction (H(\kappa))^V$. Hence we conclude no bounded subset of $\kappa$ is added after forcing with $\mathbb{P}$ and then $\kappa$ remains strongly inaccessible in $V[G]$.
\end{proof}

\begin{claim}
In $V[G]$,
 $(\forall p\in T) (\exists q \succeq p) (\exists \gamma\in \delta) (\forall q' \succeq q) \   A_{q',\gamma}\in G$.
\end{claim} 
 
 \begin{proof}[Proof of the claim]
 Note $T=j(T)\cap \kappa^{<\kappa}\in M$.
  Suppose the claim is not true, in $V[G]$, $(\exists p\in T) (\forall \gamma\in \delta) (\forall q\succeq p) (\exists q'\succeq q) \ A_{q',\gamma}^c\in G$. In $V[G]$ we build an increasing chain of nodes in $T$, $\langle p_\gamma: \gamma<\delta\rangle$ such that $p_0=p$ and $\forall \gamma\in \delta$, $A_{p_\gamma,\gamma}^c\in G$. As by Claim \ref{NoNewBoundedSubset} $\langle p_\gamma :\gamma<\delta\rangle \in V$, we can pick $q^*\in T$ that is an upper bound for $\langle p_\gamma: \gamma <\delta\rangle$ in $T$. 
However, we know $\bigcap_{\gamma<\delta} A_{q^*, \gamma}^c \in G$ by the construction and the fact that $G$ is $V$-$\kappa$-complete, $\langle p_\gamma: \gamma<\delta\rangle \in V$. Work in V. We know that $\bigcap_{\gamma<\delta} A_{q^*, \gamma}^c$ is not bounded in $\kappa$.
Let $\beta\in \bigcap_{\gamma<\delta} A_{q^*, \gamma}^c$ and $\beta>ht(q^*)$. Then there is no $q''\in T(\beta)[q^*]$ that gets color $\gamma$ for any $\gamma<\delta$.   
 This is clearly a contradiction.
 \end{proof}

Now in $V[G]$, pick the witness $q\in T, \gamma\in \delta$ for the $root(T)$. We claim that in $V$ for any $\xi > ht(q)$, there exists $\beta>\xi$ and $D\subset T(\beta)$ such that all elements in $D$ get color $\gamma$ and $D$ is $\xi$-$q$-dense. That is to say, $DSHL(1,\delta,\kappa)$ is true, which then implies $HL(1,\delta,\kappa)$ by Theorem \ref{EquiForm}. Let $\langle  q_i: i<\alpha\rangle$ enumerate $T(\xi)[q]$. We know in $V[G]$ that $\bigcap_{i<\alpha} A_{q_i,\gamma}\backslash \xi+1\in G$. In particular, it can't have empty intersection in $V$. Pick any $\beta\in \bigcap_{i<\alpha} A_{q_i, \gamma}\backslash \xi+1$, then by definition for each $i<\alpha$, there exists $q_i'\in T(\beta)[q_i]$ such that $f(q_i')=\gamma$. Thus $D=\{q_i'\in T(\beta): i<\alpha\}$ works.

\end{proof}

In fact an almost identical proof shows the asymmetric version, for any $\delta<\kappa$, $HL^{asym}(1,\delta ,\kappa)$, is true. Combine with the following theorem by Kunen: 
\begin{thm}[Kunen \cite{MR495118}]
It is consistent relative to the existence of a measurable cardinal that there exists a strongly inaccessible cardinal $\kappa$ that is not weakly compact but carries a $\kappa$-saturated $\kappa$-complete ideal.
\end{thm}

The following corollary is then immediate.

\begin{cor}
It is consistent relative to the existence of a measurable cardinal that there exists an inaccessible cardinal not weakly compact and for any $\delta<\kappa \ HL(1,\delta,\kappa)$ holds.
\end{cor}

We can obtain similar conclusions in higher dimension at the cost of stronger hypothesis.

\begin{thm}\label{Final}
Suppose $\kappa$ is inaccessible and GCH holds above $\kappa$. Moreover, assume $\kappa$ is measurable in the forcing extensions by $Add(\kappa, \kappa^{+d})$ for all $d\in \omega$.
Then there exists a forcing extension in which $\kappa$ is inaccessible but not weakly compact and for all $d\in \omega, \delta<\kappa, \ HL(d,\delta,\kappa)$ holds.
\end{thm}

\begin{proof}
By Kunen's absorption technique (see Kunen \cite{MR495118}), $Add(\kappa,1)$ is forcing equivalent with $P*\dot{Q}$ where $P$ adds a $\kappa$-Suslin tree $S$ and $\Vdash_{P}\dot{Q}$ adds a branch to $S$. Let $g$ be generic for $P$ over $V$. We claim that $V[g]$ is the desired model.

First note that $\kappa$ is inaccessible but not weakly compact in $V[g]$ as there exists a $\kappa$-Suslin tree in $V[g]$ and the forcing adds no bounded subsets of $\kappa$. Given $d\in \omega, \delta<\kappa,\langle T_j: j<d\rangle , f: \Pi_{j<d}T_j\to \delta$ as in $HL(d,\delta, \kappa)$, let $h$ be $V[g]$-generic for $(\dot{Q})_g=Q$. By assumption we know in $V[g][h]$ $\kappa$ is measurable and it remains so after further adding any $\lambda\in \{\kappa^{+d}: d\in \omega\}$ many Cohen subsets of $\kappa$. 
Furthermore, as over $V[g]$ $Q$ does not add any bounded subsets of $\kappa$, $T_j$ remains a nice tree in $V[g][h]$ for $j<d$. Hence by Theorem \ref{DHmainThm}, in $V[g][h]$, there exist $\zeta<\kappa, \bar{t}=\langle t_j: j<d \rangle\in \Pi_{j<d}T_j(\zeta)$ such that for all $\zeta'<\kappa$, there is $\zeta'' \geq \zeta', \zeta''<\kappa$ and $\langle X_j\subset T_j(\zeta''): j<d\rangle$ such that for each $j<d$, $X_j$ dominates $T_j(\zeta')[t_j]$ and $|f'' \Pi_{j<d} X_j|=1$. However, since forcing with $Q$ over $V[g]$ does not add any new $<\kappa$-sequence of ordinals, these witnesses can be found in $V[g]$. Therefore, the same statement holds in $V[g]$.

\end{proof}

\begin{remark}
The hypothesis in Theorem \ref{Final} can be forced from \linebreak
$V\models $ GCH + there exists a cardinal $\kappa$ that is $(\kappa+\omega)$-strong.
\end{remark}

\section*{Acknowledgment}
I want to thank James Cummings for illuminating discussions of the problem. I am especially grateful for his time spent on reading and correcting previous drafts. I also want to thank the anonymous referees for carefully reading the drafts and suggesting lots of improvements, including pointing out countless grammatical mistakes and spots of poor readability.

\appendix

\section{A proof of Lemma \ref{Cleanup}} \label{AddedProof}

First we show when $d=1$ we can use $\lambda=(2^\kappa)^+$ to do the construction. To see this, note we can find $W(\alpha)\in [\lambda]^{\leq \kappa}$ for each $\alpha\in \lambda$  fulfilling the first requirement  by $\kappa^+$-c.c-ness of $\mathbb{P}$. Applying the $\Delta$-system lemma, we can find $Z\in [\lambda]^\lambda$ such that $\{W(\gamma): \gamma\in Z\}$ forms a $\Delta$-system with root $R\in [\lambda]^{\leq \kappa}$ and $\forall \alpha\in Z\ W(\alpha)\cap \sup R = R$. Define $W(\emptyset)=R$.

Define an equivalence relation $\sim$ on $Z$ such that $\alpha \sim \beta$ iff 
\begin{itemize}
\item $type(W(\alpha))=type(W(\beta))$;
\item $type(\alpha\cap W(\alpha))=type(\beta\cap W(\beta))$;
\item For any $p\in \mathbb{P}\restriction W(\alpha)$, $i<\kappa$ and $\gamma<\delta_i$, we have 
\begin{equation*}
p\Vdash \dot{\tau}_i(\alpha)=\gamma \Leftrightarrow h_{W(\alpha),W(\beta)}(p) \Vdash \dot{\tau}_i(\beta)=\gamma.
\end{equation*}
\end{itemize}

 The number of equivalence classes is at most $\leq 2^\kappa$. Hence by the hypothesis there exists a $E\in [Z]^\kappa$ consisting of $\sim$-equivalent elements (in fact we can even find such a set of size $\lambda$). It then follows that $E$ and $W$ are as desired.

More generally, fix $d\in \omega$ and suppose $\lambda\to (\kappa)_{2^\kappa}^{2d}$. If $d=1$, argue as above. So consider $d\geq 2$. First we list some facts proved in \cite{MR2520110}.

\begin{fact}\label{preprocessed}
There exist $E'\in [\lambda]^\kappa$ and $\{W'(u)\in [\lambda]^{\leq \kappa}: u\in [E']^{\leq d} \}$ satisfying the following: 
\begin{enumerate}
\item \label{facta1} $u\subset W'(u)$ for all $u\in [E']^{\leq d}$;
\item \label{facta2} for any $u,v\in [E']^{\leq d}$, if $u\subset v$, then $W'(u)\subset W'(v)$;
\item \label{facta3} for each $i<\kappa$ and each $u\in [E']^d$, $\mathbb{P}\restriction W'(u)$ contains a maximal antichain deciding $\dot{\tau}_i(u)$;
\item \label{facta4} for any $u,v\in [E']^d$, $type(W'(u))=type(W'(v))$, $h_{W'(u), W'(v)}(u)=v$ and for any $i<\kappa$, $j<\delta_i$, $p\in \mathbb{P}\restriction W'(u)$, $p\Vdash \dot{\tau}_i(u) = j $ iff $h_{W'(u), W'(v)}(p)\Vdash \dot{\tau}_i(v)=j$.
\item \label{facta5} for any $k\in \omega, \{u_i: i<k\}, \{w_i: i<k\} \subset [E']^d$, if $(\bigcup_{i<k} u_i, u_0, \cdots, u_{k-1})\simeq (\bigcup_{i<k} w_i, w_0, \cdots, w_{k-1})$, then 
\begin{equation*}
(\bigcup_{i<k} W'(u_i), W'(u_0),\cdots, W'(u_{k-1}))\simeq (\bigcup_{i<k} W'(w_i), W'(w_0),\cdots, W'(w_{k-1}))
\end{equation*} 

\end{enumerate}
\end{fact}

Define for each $u\in [E']^{\leq d}$, $W(u)=\bigcup
 \{\bigcap_{v\in X} W'(v): X\subset [E']^{\leq d} \  \& \ \bigcap X \subset u\}$. Let $E=\{\gamma_{\omega \nu}: \nu<\kappa\}$ where $\{\gamma_\nu: \nu<\kappa\}$ is the increasing enumeration of $E'$. We will collect more facts about $W$ and $E$ from \cite{MR2520110}.
 
\begin{fact}\label{typel}
 \begin{enumerate}
 \item \label{factb1} For all $u, v\in [E]^{\leq d} \ W(u)\cap W(v)=W(u\cap v)$;
 \item \label{factb2} We can without loss of generality take $X$ in the definition of $W(u)$ to be of cardinality at most $d+1$;
 \item \label{factb3} For any $X=\{u_i: i<k\}\subset [E']^{\leq d}$ and $Y=\{v_i: i<k\}\subset [E']^{\leq d}$ and $u\in [E]^{\leq d}$, if $(\bigcup X, u, u_0, \cdots, u_{k-1})\simeq (\bigcup Y, u, v_0,\cdots, v_{k-1})$, then $\bigcap_{v\in X} W'(v) = \bigcap_{v\in Y} W'(v)$. We say $X,Y$ have the same \emph{isomorphism type with respect to $u$} if $\bigcap X \subset u$, $\bigcap Y\subset u$ and there exist some enumerations of $X$ and $Y$ satisfying $(\bigcup X, u, u_0, \cdots, u_{k-1})\simeq (\bigcup Y, u, v_0,\cdots, v_{k-1})$.
 \end{enumerate}
\end{fact}

It suffices to consider $X\subset [E']^d$ in the definition of $W(u)$. Since for each finite $X=\{u_0,\cdots, u_{k-1}\}$ such that $\bigcap X\subset u$, we can find $u_i'\supset u_i$ for $i<k$ such that $u_i'\in [E']^d$ and $\bigcap_{i<k} \{u_i': i<k\}=\bigcap X$. Let $X'=\{u_i': i<k\}$. Then as $W'(u_i)\subset W'(u_i')$ so $\bigcap_{v\in X} W'(v) \subset \bigcap_{v\in X'} W'(v)$.

Now we are ready to show that $W$ and $E$ satisfy the requirements of Lemma \ref{Cleanup}. \ref{CL3} is satisfied by (\ref{factb1}) in Fact \ref{typel}. First we verify requirement \ref{CL4}. Given $u_2, w_2\in [E]^d$ and $u_1\subset u_2$ and $w_1\subset w_2$ with $(u_2,u_1,<)\simeq (w_2,w_1, <)$, we need to show 
\begin{equation*}h_{W(u_2), W(w_2)}(W(u_1))=W(w_1).
\end{equation*}
Let $\{X_p\subset [E']^d : p<l\}$ enumerate the representatives for the isomorphism types with respect to $u_2$ with $|X_p|\leq d+1$. Enumerate each $X_p$ as $\{u_p^i : i<k_p\}$. More precisely, for any $X\subset [E']^d$ and $|X|=k\leq d+1$ such that $\bigcap X \subset u_2$, there exists some $p<l$ such that $X$ and $X_p$ have the same isomorphism type in the sense of \ref{factb3} in Fact \ref{typel}.

For each $X_p$ we can find isomorphic $X_p'=\{w_p^i: i<k_p\}\subset [E']^d$, in the sense that  \begin{equation*}
(\bigcup X_p, u_2, u_p^0, \cdots, u_p^{k_p-1})\simeq (\bigcup X_p', w_2, w_p^0, \cdots, w_p^{k_p-1}),
\end{equation*}
 so that $\{X_p': p<l\}$ will be an enumeration of the representatives for the isomorphism types with respect to $w_2$, satisfying that \begin{equation*}
\begin{split}
 (\bigcup_{p<l ,i<k_p} u_p^i, u_2, u_0^0, \cdots , u_0^{k_0-1}, \cdots , u_{l-1}^{0}, \cdots , u_{l-1}^{k_{l-1}-1})\simeq \\  (\bigcup_{p<l ,i<k_p} w_p^i,w_2, w_0^0, \cdots , w_0^{k_0-1}, \cdots , w_{l-1}^{0}, \cdots , w_{l-1}^{k_{l-1}-1}).
 \end{split}
\end{equation*} 
 
The reason why we can do this is $u_2$ and $w_2$ only contains limit points with respect to the increasing enumeration of $E'$. Hence when choosing $X_p'$ for $p<l$, we make sure there are sufficiently many elements in $E'$ between consecutive elements in $\bigcup_{s\leq p} ( \bigcup X_s')\cup w_2$. Figure ~\ref{Demo} is a demonstration.
\begin{figure}

\begin{tikzpicture}
\filldraw [dotted] (0,2) --(2,2) circle (2pt) node[align=left,   below] {$a$} -- (3,2) circle (2pt) node[align=center, below] {$x_0$}-- (4,2) circle (2pt) node[align=center, below] {$x_4$} --(5,2) circle (2pt) node[align=center, below] {$x_1$} --
(6,2) circle (2pt) node[align=center, below] {$b=x_2$}     -- (8,2) circle (2pt) node[align=center, below] {$x_5$} -- (9,2) circle (2pt) node[align=center, below] {$x_6$} -- 
(10,2) circle (2pt) node[align=right,  below] {$c$}
;
\filldraw [dotted] (0,0) circle (2pt) node[align=left,   below] {$d$} -- (1,0) circle (2pt) node[align=center, below] {$x_0'$}-- (2,0) circle (2pt) node[align=center, below] {$x_4'$} --(3,0) circle (2pt) node[align=center, below] {$x_1'$}--
(4,0) circle (2pt) node[align=center, below] {$e=x_2'$}     -- (6,0) circle (2pt) node[align=center, below] {$x_5'$} -- (7,0) circle (2pt) node[align=center, below] {$x_6'$} --
(8,0) circle (2pt) node[align=right,  below] {$f$}
-- (10, 0)
;
\draw [->, thick] (2,2) -- (0.05,0.05);
\draw [->, thick] (6,2) -- (4.05,0.05);
\draw [->, thick] (10,2) -- (8.05,0.05);

\draw [->, thick] (3,2) -- (1.05,0.05);
\draw [->, thick] (4,2) -- (2.05,0.05);
\draw [->, thick] (5,2) -- (3.05,0.05);
\draw [->, thick] (8,2) -- (6.05,0.05);
\draw [->, thick] (9,2) -- (7.05,0.05);

\end{tikzpicture}
\caption{In the case where $d=3$, $u_2=\{a<b<c\}$, $w_2=\{d<e<f\}$, $X_0=\{\{x_0<x_1<x_2\}, \{x_4<x_5<x_6\}\}$ where $a<x_0 <x_4<x_1<b=x_2<x_5<x_6<c$, the graph shows how we find $X_0'=\{\{x_0'<x_1'<x_2'\}, \{x_4'<x_5'<x_6'\}\}$. We will make sure the length of the gap, namely the number of elements in $E'$, between consecutive elements in $(\bigcup X_0')\bigcup w_2$ is at least $(d(d+1))^l$. Recall $l$ is the number of representatives for the isomorphism types with respect to $u_2$. In general, we make sure that the number of elements in $E'$ between consecutive elements in $\bigcup_{s\leq p} (\bigcup X_s') \cup w_2$ is at least $(d(d+1))^{(l-p)}$.
The reason for the choice of the gap length is that we need to make sure there is enough room between consecutive elements in $\bigcup_{s\leq p} (\bigcup X_s') \cup w_2$ so that there is no problem when picking elements in $X_{p+1}'$.}
\label{Demo}
\end{figure}
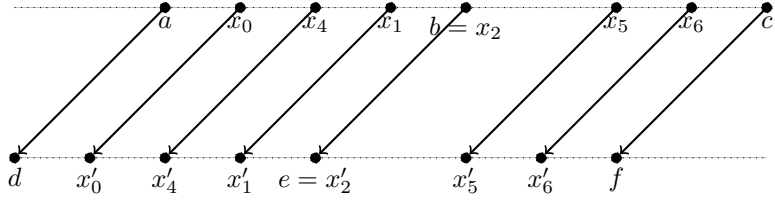

By (\ref{facta5}) in Fact \ref{preprocessed} this implies 

\begin{equation*}
\begin{split}
 (\bigcup_{p<l ,i<k_p} W'(u_p^i), W'(u_2), W'(u_0^0), \cdots , W'(u_0^{k_0-1}), \cdots , W'(u_{l-1}^{0}), \cdots , W'(u_{l-1}^{k_{l-1}-1}))\simeq \\  (\bigcup_{p<l ,i<k_p} W'(w_p^i), W'(w_2), W'(w_0^0), \cdots , W'(w_0^{k_0-1}), \cdots , W'(w_{l-1}^{0}), \cdots , W'(w_{l-1}^{k_{l-1}-1})).
\end{split}
\end{equation*}
Let $h$ be the isomorphism. Then $h\restriction W(u_2)=\bigcup \{\bigcap_{v\in X_p} W'(v): p<l\}$ is the isomorphism with $W(w_2)=\bigcup \{\bigcap_{v\in X_p'} W'(v): p<l\}$. 
Notice that $h$ respects blocks, more precisely, for each $p<l$, $h(\bigcap_{v\in X_p} W'(v))= \bigcap_{v\in X_p'} W'(v)$. Since there exists $K\subset l$ such that $W(u_1)=\bigcup \{\bigcap_{v\in X_p} W'(v): p\in K\}$ and $W(w_1)=\bigcup \{\bigcap_{v\in X_p'} W'(v): p\in K\}$ (exactly those $p<l$ such that $\bigcap X_p \subset u_1$ or equivalently $\bigcap X_p' \subset w_1$), so $h(W(u_1))=W(w_1)$.
This verifies requirement \ref{CL4}.

This also shows that the isomorphism extends that of $W'(u_2)\to W'(w_2)$ so $h(u_2)= w_2$. In particular requirement \ref{CL1} of Lemma \ref{Cleanup} is satisfied. To see requirement \ref{CL2} is satisfied, given $i<\kappa, \gamma<\delta_i, p\in \mathbb{P}\restriction W(u_2)$, such that $p\Vdash \dot{\tau}_i(u_2)=\gamma$. We need to show $h(p)=h_{W(u_2), W(w_2)}(p)\Vdash \dot{\tau}_i(w_2)=\gamma$. Suppose not, there exists $q'\leq h(p)$ and $l\neq \gamma, l<\delta_i$ such that $q'\Vdash \dot{\tau}_i(w_2)=l$. Since $\mathbb{P}\restriction W'(w_2)$ contains a maximal antichain deciding the values of $\dot{\tau}_i(w_2)$, there exists $q\in \mathbb{P}\restriction W'(w_2)$ such that $q\leq q'\restriction W'(w_2)$ that $q\Vdash \dot{\tau}_i(w_2)=l$. By the choice of $W'(w_2)$ as in (\ref{facta4}) of Fact \ref{preprocessed}, we have $h_{W'(u_2),W'(w_2)}^{-1} (q)=h_{W(u_2),W(w_2)}^{-1} (q)=h^{-1}(q)\Vdash \dot{\tau}_i(u_2)=l$. But $h^{-1}(q)\leq p\restriction W'(u_2)$ and $h^{-1}(q)\in \mathbb{P}\restriction W'(u_2)$. So take any $q^*\leq h^{-1}(q), p\restriction (W(u_2)-W'(u_2))$, then $q^*\leq p$ and $q^*\Vdash \dot{\tau}_i(u_2)=l$, contradiction. To get the other direction, argue with $h^{-1}$ in place of $h$.

\bibliographystyle{plain}
\bibliography{citation}

\begin{thebibliography}{10}

\bibitem{MR2768691}
James Cummings.
\newblock Iterated forcing and elementary embeddings.
\newblock In {\em Handbook of set theory. {V}ols. 1, 2, 3}, pages 775--883.
  Springer, Dordrecht, 2010.

\bibitem{Devlin}
Denis Devlin.
\newblock {\em Some partition theorems and ultrafilters on $\omega$}.
\newblock PhD Thesis.

\bibitem{DobrinenHathaway2}
Natasha Dobrinen and Daniel Hathaway.
\newblock Forcing and the halpern-läuchli theorem.
\newblock {\em preprint}.
\newblock \url{https://arxiv.org/abs/1706.08174}.

\bibitem{DobrinenHathaway}
Natasha Dobrinen and Daniel Hathaway.
\newblock The halpern-läuchli theorem at a measurable cardinal.
\newblock {\em The Journal of Symbolic Logic}.
\newblock to appear.

\bibitem{MR2520110}
M.~D\v{z}amonja, J.~A. Larson, and W.~J. Mitchell.
\newblock A partition theorem for a large dense linear order.
\newblock {\em Israel J. Math.}, 171:237--284, 2009.

\bibitem{MR2768692}
Matthew Foreman.
\newblock Ideals and generic elementary embeddings.
\newblock In {\em Handbook of set theory. {V}ols. 1, 2, 3}, pages 885--1147.
  Springer, Dordrecht, 2010.

\bibitem{MR1606036}
A.~Hajnal and P.~Komj{\'a}th.
\newblock A strongly non-{R}amsey order type.
\newblock {\em Combinatorica}, 17(3):363--367, 1997.

\bibitem{MR0200172}
J.~D. Halpern and H.~L\"auchli.
\newblock A partition theorem.
\newblock {\em Trans. Amer. Math. Soc.}, 124:360--367, 1966.

\bibitem{MR0284328}
J.~D. Halpern and A.~L\'evy.
\newblock The {B}oolean prime ideal theorem does not imply the axiom of choice.
\newblock In {\em Axiomatic {S}et {T}heory ({P}roc. {S}ympos. {P}ure {M}ath.,
  {V}ol. {XIII}, {P}art {I}, {U}niv. {C}alifornia, {L}os {A}ngeles, {C}alif.,
  1967)}, pages 83--134. Amer. Math. Soc., Providence, R.I., 1971.

\bibitem{MR495118}
Kenneth Kunen.
\newblock Saturated ideals.
\newblock {\em J. Symbolic Logic}, 43(1):65--76, 1978.

\bibitem{MR754925}
Richard Laver.
\newblock Products of infinitely many perfect trees.
\newblock {\em J. London Math. Soc. (2)}, 29(3):385--396, 1984.

\bibitem{MR0491197}
A.~R.~D. Mathias.
\newblock Happy families.
\newblock {\em Ann. Math. Logic}, 12(1):59--111, 1977.

\bibitem{MR1218224}
S.~Shelah.
\newblock Strong partition relations below the power set: consistency; was
  {S}ierpi{\'n}ski right? {II}.
\newblock In {\em Sets, graphs and numbers ({B}udapest, 1991)}, volume~60 of
  {\em Colloq. Math. Soc. J{\'a}nos Bolyai}, pages 637--668. North-Holland,
  Amsterdam, 1992.

\bibitem{MR2355670}
Stevo Todorcevic.
\newblock {\em Walks on ordinals and their characteristics}, volume 263 of {\em
  Progress in Mathematics}.
\newblock Birkh\"auser Verlag, Basel, 2007.

\bibitem{MR1486583}
S.~Todor\v{c}evi\'{c} and I.~Farah.
\newblock {\em Some applications of the method of forcing}.
\newblock Yenisei Series in Pure and Applied Mathematics. Yenisei, Moscow;
  Lyc\'ee, Troitsk, 1995.

\bibitem{MR2603812}
Stevo Todor\v{c}evi\'{c}.
\newblock {\em Introduction to {R}amsey spaces}, volume 174 of {\em Annals of
  Mathematics Studies}.
\newblock Princeton University Press, Princeton, NJ, 2010.

\end{thebibliography}

\Addresses

\end{document}